\documentclass[11pt,thmsa]{article}

\usepackage{amssymb,latexsym}
\usepackage{amsmath,amscd}
\usepackage[all]{xy}

 \usepackage[latin1]{inputenc}
 \usepackage{mathrsfs}
 \usepackage{dsfont}
 \usepackage{amsmath}
 \usepackage{amsthm}
 \usepackage{amsfonts}
 \usepackage{amssymb}
 \usepackage[dvips]{graphicx}
 \usepackage{wrapfig}
 \usepackage{layout}
 \usepackage{verbatim}
 \usepackage{alltt}
 \usepackage{xypic}
 \input xy
 \xyoption{all}

 \DeclareMathAlphabet{\mathpzc}{OT1}{pzc}{m}{it}

 \newtheorem{theorem}{Theorem}[section]
 \newtheorem{lemma}[theorem]{Lemma}
 \newtheorem{proposition}[theorem]{Proposition}
 
 \newtheorem{corollary}[theorem]{Corollary}
 \newtheorem{definition}[theorem]{Definition}
  \theoremstyle{definition}
 \newtheorem{example}[theorem]{Example}
 \newtheorem{remark}[theorem]{Remark}

\newtheorem{free}{}

\newtheorem*{acknowledgements}{Acknowledgements}
\renewenvironment{proof}{\noindent{\it
Proof.}}{\bgroup\hspace{\stretch{1}}$\square$\egroup\medskip\par}
\newcommand{\rmap}{\longrightarrow}
\newcommand{\Rep}{\textrm{Rep}}
\newcommand{\Der}{\textrm{Der}}
\newcommand{\RRep}{\mathbb{R}\textrm{ep}}
\newcommand{\DDer}{\mathbb{D}\textrm{er}}

\begin{document}
\vspace{15cm}
 \title{Representations up to homotopy of Lie algebroids}
\author{Camilo Arias Abad\footnote{Institut f\"ur Mathematik, Universit\"at Z\"urich,
camilo.arias.abad@math.uzh.ch. Partially supported by NWO Grant ``Symmetries and Deformations in Geometry'' and
by SNF Grant 200020-121640/1. } and Marius Crainic\footnote{Mathematics Institute, Utrecht University, m.crainic@uu.nl. Partially supported by NWO Vidi Grant no. 639.032.712.}}
 \maketitle
 \abstract{
We introduce and study the notion of
representation up to homotopy of a Lie algebroid, paying special attention to examples. We
use representations up to homotopy to define the adjoint representation of a Lie algebroid and
show that the resulting cohomology controls the deformations of the structure. The Weil algebra
of a Lie algebroid is defined and shown to coincide with Kalkman's BRST model for equivariant cohomology in the case
of group actions. The relation of this algebra with the integration of Poisson and Dirac structures is explained in \cite{AC2}.

\section{Introduction}

Lie algebroids are infinite dimensional Lie algebras which can be
thought of as  \textit{generalized tangent bundles} associated to
various geometric situations. Apart from Lie algebras and 
tangent bundles, examples of Lie algebroids come from
foliation theory, equivariant geometry, Poisson geometry, riemannian
foliations, quantization, etc. Lie algebroids are the infinitesimal
counterparts of Lie groupoids exactly in the same way in which Lie
algebras are related to Lie groups. Generalizing representations
of Lie algebras as well as vector bundles endowed with a flat connection,
a representation of a Lie algebroid on a vector bundle is 
an action by derivations on the space of sections.

The aim of this paper is to introduce and study a more general notion of
representation: ``representations up to homotopy''.
Our approach is based on
Quillen's superconnections \cite{Q} and fits into the
general theory of structures up to homotopy. The idea is to represent Lie algebroids in
cochain complexes of vector bundles, rather than in vector bundles. 
In a representation up to homotopy the complex is given an action of the Lie algebroid and homotopy operators $\omega_k$
 for $k\geq 2$. The action is required to be flat only up to homotopy, that is,
the curvature of the action may be non-zero, but it is homotopic to
zero via the homotopy $\omega_2$; in turn, $\omega_2$ is required to satisfy the
appropriate coherence condition (Bianchi identity) only up to a homotopy
given by $\omega_3$, and there are higher order coherence conditions. 
Quillen's formalism is used to book-keep the
equations involved.

The advantage of considering these
representations is that they are flexible and general enough to
contain interesting examples which are the correct generalization of
the corresponding notions for Lie algebras. In the setting of representations up to homotopy one can
give a good definition of the adjoint representation of a Lie algebroid.
We will show that, as in the case of Lie algebras, the cohomology associated to the adjoint representation of a Lie algebroid
controls the deformations of the structure (it coincides with the deformation cohomology of \cite{CM}). There are other seemingly ad-hoc
equations that arise from various geometric problems and can now be recognized as cocycle equations
for the cohomology associated to a representation up to homotopy. This is explained in Proposition \ref{proposition k-differentials} for the $k$-differentials
of  \cite{IG}.

Our original motivation for considering representations up to homotopy is
the study of the cohomology of classifying spaces of Lie groupoids.  We will
introduce the Weil algebra $W(A)$ associated to a Lie algebroid $A$. When the Lie algebroid
comes from a Lie groupoid $G$, the Weil algebra serves as a model for the De-Rham
algebra of the total space of the universal principal $G$-bundle $EG\rmap BG$. 
As we shall explain, this generalizes not only the usual Weil algebra of a Lie algebra
but also Kalkman's BRST algebra for equivariant cohomology. Further applications
of the Weil algebra are given in \cite{AC2}, while applications of the notion of representations
up to homotopy to the cohomology of classifying spaces -Bott's spectral sequence- are explained in \cite{AC3}.

A few words about the relationship of this paper with other work that we found in the literature.
Representations up to homotopy can also be described in the language of differential graded
modules over differential graded algebras \cite{HMS}. We emphasize however that we insist on working
with DG-modules which are sections of vector bundles. Our Weil algebra is
isomorphic to the one given in Mehta's thesis \cite{Mehta},
where, using the language of supermanifolds, it appears as 
$(C^{\infty}([-1]T([-1]A)),{\cal{L}}_{d_A}+d)$. Similar descriptions were communicated to us
by D. Roytenberg and P. Severa (unpublished). Some of the equations that appear in the definition of the adjoint
representation were considered by Blaom in \cite{Blaom}. Representations up to homotopy provide examples of
the $Q$-bundles studied by Kotov and Strobl in \cite{KS}.

This paper is organized as follows.
Section \S\ref{Representations and cohomologies} begins by collecting
the definitions of Lie algebroids, representations
and the associated cohomology theories. Then, the connections and curvatures
underlying the adjoint representation are described.

In Section \S\ref{Representations up to homotopy} we give the
definition of representations up to homotopy, introduce several
examples and explain the relationship with extensions
(Proposition \ref{representations and
extensions}).
We define the adjoint representation
and point out that the associated cohomology is isomorphic to the deformation
cohomology of \cite{CM} (Theorem \ref{The adjoint representation}).
We also explain the relation between the adjoint representation and the
first jet algebroid (Proposition \ref{cor-first-jet}).
 In the case of the tangent bundle we describe the associated
parallel transport (Proposition \ref{parallel}).

In Section \S\ref{Operations, cohomology and the derived category} we discuss the main
operations on the category of representations up to homotopy and then we have a closer look
at the resulting derived category. In particular, we prove that our notion has the usual properties that
one expects for ``structures up to homotopy''  (Proposition
\ref{qi's} and Theorem \ref{rep-regular}). Various examples are presented.

In Section \S\ref{The Weil algebra} we introduce 
the Weil algebra of a Lie algebroid, generalizing the
standard Weil algebra of a Lie algebra. We show that, when applying
this construction to the Lie algebroid associated to a Lie group
action on a manifold, one obtains Kalkman's BRST algebra for
equivariant cohomology
(Proposition \ref{Weil and BRST}).

In order to fix our conventions, we recall
some basic properties of graded algebra and complexes of vector bundles in the appendix.

\begin{acknowledgements}
We would like to thank the referee for carefully reading a previous version of this manuscript and suggesting many improvements. We would
also like to thank Dmitry Roytenberg for several discussions we had and James Stasheff for valuable comments. 
\end{acknowledgements}

\section{Preliminaries}\label{Representations and cohomologies}

\subsection{Representations and cohomology}


Here we review some standard facts about Lie algebroids,
representations and cohomology. We also make a few
 general comments regarding the notion of adjoint
representation for Lie algebroids. Throughout this paper, $A$ denotes a Lie algebroid over a fixed base
manifold $M$. As a general reference for algebroids, we use
\cite{McK}.

\begin{definition}
A Lie algebroid over $M$ is a vector bundle $\pi:A \rightarrow M$
together with a bundle map $\rho:A \rightarrow TM$, called the
anchor map, and a Lie bracket on the space $\Gamma(A)$ of sections
of $A$ satisfying the Leibniz identity:
$$[\alpha,f\beta]=f[\alpha,\beta]+\rho(\alpha)(f)\beta,$$
for every $\alpha,\beta \in \Gamma(A)$ and $f \in C^{\infty}(M)$.
\end{definition}

Given an algebroid $A$, there is an associated De-Rham complex
$\Omega(A)= \Gamma(\Lambda A^*)$, with the De-Rham operator given by
the Koszul formula
\begin{eqnarray*}
d_A \omega(\alpha_1, \dots ,\alpha_{k+1})&=& \sum_{i<j}(-1)^{i+j} \omega([\alpha_i,\alpha_j],\dots,\hat{\alpha_i},\dots,\hat{\alpha_j},\dots,\alpha_{k+1})\\
&&+\sum_i(-1)^{i+1}L_{\rho(\alpha_i)}\omega(\alpha_1,\dots,\hat{\alpha_i},\dots,\alpha_{k+1}),
\end{eqnarray*}
where $L_{X}(f)= X(f)$ is the Lie derivative along vector fields.
The operator $d_A$ is a differential ($d_{A}^{2}= 0$) and satisfies
the derivation rule
\[ d_A(\omega\eta)= d_A(\omega)\eta+ (-1)^p\omega d_A(\eta),\]
for all $\omega\in \Omega^p(A), \eta\in \Omega^q(A)$.

\begin{definition}
Let $A$ be a Lie algebroid over $M$. An $A$-connection on a vector
bundle $E$ over $M$ is an $\mathbb{R}$-bilinear map $\nabla:\Gamma(A)\times
\Gamma(E) \rightarrow \Gamma(E)$, $(\alpha,S) \mapsto
\nabla_{\alpha}(S)$ such that:
\[  \nabla_{f\alpha}(s) = f\nabla_{\alpha}(s),\ \ \nabla _{\alpha}(fs)=f\nabla_{\alpha}(s)+L_{\rho(\alpha)}(f)(s)\]
for all $f\in C^{\infty}(M)$, $s\in \Gamma(E)$ and $ \alpha \in
\Gamma(A)$. The $A$-curvature of $\nabla$ is the tensor given by
\begin{equation*}
R_{\nabla}(\alpha, \beta)(s):= \nabla_{\alpha}\nabla_{\beta}(s)-
\nabla_{\beta}\nabla_{\alpha}(s)- \nabla_{[\alpha, \beta]}(s)
\end{equation*}
 for $\alpha,\beta \in \Gamma(A)$, $s \in \Gamma(E)$.
The $A$-connection $\nabla$ is called flat if $R_{\nabla}= 0$. A
representation of $A$ is a vector bundle $E$ together with a flat
$A$-connection $\nabla$ on $E$.
\end{definition}

Given any $A$-connection $\nabla$ on $E$, the space of $E$-valued
$A$-differential forms, $\Omega(A; E)= \Gamma(\Lambda A^*\otimes E)$
has an induced operator $d_{\nabla}$ given by the Koszul formula
\begin{eqnarray*}
d_{\nabla} \omega(\alpha_1, \dots ,\alpha_{k+1})&=& \sum_{i<j}(-1)^{i+j} \omega([\alpha_i,\alpha_j],
\dots,\hat{\alpha_i},\dots,\hat{\alpha_j},\dots,\alpha_{k+1})+\\
&+&\sum_i(-1)^{i+1}\nabla_{\alpha_i}\omega(\alpha_1,\dots,\hat{\alpha_i},\dots,\alpha_{k+1}).
\end{eqnarray*}
In general, $d_{\nabla}$ satisfies the derivation rule
\[ d_{\nabla}(\omega\eta)= d_A(\omega)\eta+ (-1)^p\omega d_{\nabla}(\eta),\]
and squares to zero if and only if $\nabla$ is flat.

\begin{proposition}\label{connections-derivations} Given a Lie algebroid $A$ and
a vector bundle $E$ over $M$, there is a 1-1 correspondence
between $A$-connections $\nabla$ on $E$ and degree $+1$ operators
$d_{\nabla}$ on $\Omega(A; E)$ which satisfy the derivation
rule. Moreover, $(E, \nabla)$ is a representation if and only if
$d_{\nabla}^2= 0$.
\end{proposition}

In a more algebraic language, every Lie algebroid $A$ has an
associated DG algebra $(\Omega(A), d_A)$, and every representation
$E$ of $A$ gives a DG module over this DG algebra.

\begin{definition} Given a representation $E= (E, \nabla)$ of $A$, the cohomology of $A$ with coefficients in $E$,
denoted $H^{\bullet}(A;E)$, is the cohomology of the complex
$(\Omega(A;E),d_{\nabla})$. When $E$ is the trivial representation
(the trivial line bundle with $\nabla_{\alpha}= L_{\rho(\alpha)}$),
we write $H^{\bullet}(A)$.
\end{definition}

\begin{example}\rm \ \label{zeroth-examples}
In the extreme case where $A$ is $TM$, representations are
flat vector bundles over $M$, while the associated cohomology is
the usual cohomology of $M$ with local coefficient given by the
flat sections of the vector bundle. At the other extreme, when $A= \mathfrak{g}$ is a Lie algebra,
one recovers the standard notion of representation of Lie algebras,
and Lie algebra cohomology. 
For a foliation $\mathcal{F}$ on $M$, viewed as an involutive sub-bundle
of $TM$, the Lie algebroid cohomology becomes the well-known
foliated cohomology (see e.g. \cite{KT,Alv}). Central to foliation theory
is the Bott  connection \cite{Bott0} on the normal bundle $\nu= TM/\mathcal{F}$,
\[ \nabla_{V}(X\  \textrm{mod}\,\mathcal{F})= [V, X]\ \textrm{mod}\,\mathcal{F} \ \ (V\in \Gamma(\mathcal{F}))\]
which is the linearized version of the notion of holonomy. In our language, $\nu$ is a
representation of $\mathcal{F}$. 

More generally, for any regular Lie algebroid $A$ (regular in the sense that $\rho: A \to TM$ has constant rank),
$A$ has two canonical representations. They are the kernel of
$\rho$, denoted $\mathfrak{g}(A)$, with the $A$-connection
\[ \nabla^{\textrm{adj}}_{\alpha}(\beta)= [\alpha, \beta],\]
and the normal bundle $\nu(A)= TM/\rho(A)$ of the foliation induced
by $A$, with the connection
\[ \nabla^{\textrm{adj}}_{\alpha}(\overline{X})= \overline{[\rho(\alpha), X]},\]
where $\overline{X}= X\  \textrm{mod}\,\rho(A)$. 
\end{example}

\subsection{Deformation cohomology}\label{subsection:deformations}

The deformation cohomology of $A$ arises in the study of the
deformations of the Lie algebroid structure \cite{CM}. This
cohomology cannot, in general, be realized as the cohomology
associated to a representation. The deformation complex
$(C^{\bullet}_{\textrm{def}}(A), \delta)$ is defined as follows. In
degree $k$, it consists of pairs $(c,\sigma_c)$ where $c$ is an antisymmetric, $\mathbb{R}$-multilinear
map
\[ c: \underbrace{\Gamma(A) \times \cdots \times
\Gamma(A)}_{k-\mathrm{times}} \rightarrow \Gamma(A)
\]
and $\sigma_c$ is an antisymmetric, $C^{\infty}(M)$-multilinear
map
\begin{equation*}
\sigma_{c}: \underbrace{\Gamma(A) \times \cdots \times
\Gamma(A)}_{(k-1)-\mathrm{times}} \rightarrow \Gamma(TM),
\end{equation*}
which is the symbol of $c$ in the sense that:
\begin{equation*}
c(\alpha_1, \dots, f \alpha_{k})=fc(\alpha_1, \dots,\alpha_{k})+
L_{\sigma_{c}(\alpha_1, \dots,\alpha_{k-1})}(f)\alpha_{k},
\end{equation*}
for any function $f \in C^{\infty}(M)$ and sections $\alpha_i \in
\Gamma(A)$. The differential
\begin{equation*}
\delta: C^k_{\textrm{def}}(A) \rightarrow C^{k+1}_{\textrm{def}}(A)
\end{equation*}
associates to $(c, \sigma_c)$ the pair $(\delta(c),\sigma_{\delta(c)})$ where:
\begin{eqnarray*}
\delta(c)(\alpha_1,\dots,\alpha_{k+1})&=&\sum_{i<j} (-1)^{i+j} c([\alpha_i,\alpha_j],\alpha_1,
\dots, \hat{\alpha_i}, \dots, \hat{\alpha_j}, \dots,\alpha_{k+1})\\
&&+\sum_{i=1}^{k+1} (-1)^{i+1}[ \alpha_i, c(\alpha_1, \dots,
\hat{\alpha_i}, \dots,\alpha_{k+1})],
\end{eqnarray*}

\[\sigma_{\delta(c)}=\delta(\sigma_c)+(-1)^{k+1}\rho \circ c,\]
and
\begin{eqnarray*}
\delta(\sigma_c)(\alpha_1,\dots,\alpha_{k})&=&\sum_{i<j} (-1)^{i+j} \sigma_c([\alpha_i,\alpha_j],\alpha_1,
\dots, \hat{\alpha_i}, \dots, \hat{\alpha_j}, \dots,\alpha_{k})\\
&&+\sum_{i=1}^{k} (-1)^{i+1}[ \rho(\alpha_i), c(\alpha_1, \dots,
\hat{\alpha_i}, \dots,\alpha_{k})].
\end{eqnarray*}
\begin{definition} The deformation cohomology of the Lie algebroid $A$, denoted
$H^{\bullet}_{\mathrm{def}}(A)$, is defined as the cohomology of
the cochain complex $(C^{\bullet}_{\mathrm{def}}(A), \delta)$.
\end{definition}

\begin{example}\rm \ \label{first-examples}
When $A= TM$, the deformation cohomology is
trivial (cf. \cite{CM}, Corollary 2). When $A= \mathfrak{g}$ is a Lie algebra, 
the deformation cohomology
is isomorphic to $H^{\bullet}(\mathfrak{g}; \mathfrak{g})$, the
cohomology with coefficients in the adjoint representation. This is
related to the fact that deformations of the Lie algebra
$\mathfrak{g}$ are controlled by $H^2(\mathfrak{g}; \mathfrak{g})$.

In the case of a foliation $\mathcal{F}$, $H^{\bullet}_{\textrm{def}}(\mathcal{F})$
is isomorphic to $H^{\bullet -1}(\mathcal{F}; \nu)$, the cohomology with coefficients in the Bott representation
(\cite{CM}, Proposition $4$). This was already explained in the work of 
Heitsch \cite{Hei} on deformations of foliations where he shows that 
such deformations are controlled
by $H^1(\mathcal{F}; \nu)$. Due to the analogy with Lie algebras, it seems natural to
declare $\nu[-1]$ -the graded vector bundle which is
$\nu$ concentrated in degree one- 
as the adjoint representation of
$\mathcal{F}$.
\end{example}

\begin{remark}\rm \ \label{adj-rep-comments}
The notion of adjoint representation will be properly defined in
Subsection \ref{subsection: The adjoint representation}. 
For now, we would like to explain why it has to be defined in the setting of representations up to homotopy.
With the examples of Lie algebras and foliations in mind, it is tempting 
to consider the natural representations of $A$ (the $\mathfrak{g}(A)$ and $\nu(A)$
of Example \ref{zeroth-examples}) and to define the adjoint representation of $A$ as
\begin{equation}
\label{question} \textrm{ad}= \mathfrak{g}(A)\oplus \nu(A)[-1],
\end{equation}
a (graded) representation with  $\mathfrak{g}(A)$ in degree zero and $\nu(A)$ in degree one. 
Even under the assumption that $A$ is regular (so that the bundles involved are smooth),
the behavior of the deformation cohomology shows that one should be more careful. Indeed, based on
the examples of Lie algebras, one expects the cohomology associated to the adjoint representation
to coincide with the deformation cohomology.
While there is a long exact sequence (\cite{CM},Theorem $3$ )  
\begin{equation}\label{les-def}
\cdots \rightarrow H^n(A;\mathfrak{g}(A))\rightarrow
H^{n}_{\textrm{def}}(A)\rightarrow H^{n-1}(A;\nu(A))
\stackrel{\delta}{\rightarrow} H^{n+1}(A;\mathfrak{g}(A))\rightarrow
\cdots,
\end{equation}
it is not difficult to find examples for which the connecting map $\delta$
is non-zero. As we shall see,
definition (\ref{question}) can be made correct provided we endow the right hand side 
with the structure of a representation up to homotopy.

The situation is  worse in the non-regular case, when
the graded direct sum $\mathfrak{g}(A)\oplus \nu(A)[1]$ is non longer smooth. 
One can overcome this by interpreting the direct sum as the cohomology of a cochain complex of vector
bundles, concentrated in two degrees:
\begin{equation}
\label{zero-adjoint} A\stackrel{\rho}{\to} TM .
\end{equation}
We will call this {\it the adjoint complex of $A$}. The idea of
using this complex in order to make sense of the adjoint
representation appeared already in \cite{ELW}, and is also present
in \cite{CF2}. However, the presence of
the extra-structure of a representation up to homotopy had been overlooked.
\end{remark}

\subsection{Basic connections and the basic curvature}
\label{Basic connections and the basic curvature}

Keeping in mind our discussion on what the adjoint representation
should be, we want to \textit{extend}
the canonical flat $A$-connections $\nabla^{\textrm{adj}}$ (from
$\mathfrak{g}(A)$ and $\nu(A)$) to $A$ and $TM$ or, even better, to
the adjoint complex (\ref{zero-adjoint}). This construction already
appeared in the theory of secondary characteristic classes
\cite{CF2} and was also used in \cite{CF1}.

\begin{definition} Given a Lie algebroid $A$ over $M$ and a connection $\nabla$ on the vector bundle
$A$, we define
\begin{enumerate}
\item The basic $A$-connection induced by $\nabla$ on $A$:
\[\nabla^{\textrm{bas}}_{\alpha}(\beta)= \nabla_{\rho(\beta)}(\alpha)+ [\alpha, \beta].\]
\item The basic $A$-connection induced by $\nabla$ on $TM$:
\[\nabla^{\textrm{bas}}_{\alpha}(X)= \rho(\nabla_{X}(\alpha))+ [\rho(\alpha), X].\]
\end{enumerate}
\end{definition}

Note that $\nabla^{\textrm{bas}}\circ \rho= \rho \circ
\nabla^{\textrm{bas}}$, i.e. $\nabla^{\textrm{bas}}$ is an
$A$-connection on the adjoint complex (\ref{zero-adjoint}). On the
other hand, the existence of a connection $\nabla$ such that
$\nabla^{\textrm{bas}}$ is flat is a very restrictive condition on
$A$. It turns out that the curvature of $\nabla^{\textrm{bas}}$
hides behind a more interesting tensor- and that is what we will
call the basic curvature of $\nabla$.

\begin{definition} Given a Lie algebroid $A$ over $M$ and a connection $\nabla$ on the vector bundle
$A$, we define the basic curvature of $\nabla$, as the tensor
\[ R_{\nabla}^{\textrm{bas}}\in \Omega^2(A; \textrm{Hom}(TM, A))\]
given by
\[ R_{\nabla}^{\textrm{bas}}(\alpha,\beta)(X):=\nabla_X([\alpha,\beta])-[\nabla_X(\alpha),\beta]-[\alpha,\nabla_X(\beta)]-\nabla_{\nabla_{\beta}^{\textrm{bas}}X}(\alpha)
+\nabla_{\nabla_{\alpha}^{\textrm{bas}}X}(\beta),
\]
where $\alpha, \beta $ are sections of $A$ and $X$ is a vector
field on $M$.
\end{definition}

This tensor appears when one looks at the curvatures of the
$A$-connections $\nabla^{\textrm{bas}}$. One may think of
$R_{\nabla}^{\textrm{bas}}$ as the expression
$\nabla_X([\alpha,\beta])-[\nabla_X(\alpha),\beta]-[\alpha,\nabla_X(\beta)]$
which measures the derivation property of $\nabla$ with respect to
$[\cdot, \cdot]$, \textit{corrected} so that it becomes
$C^{\infty}(M)$-linear on all arguments.

\begin{proposition}\label{properties-basic} For any connection $\nabla$ on $A$, one has:
\begin{enumerate}
\item The curvature of the $A$-connection $\nabla^{\textrm{bas}}$
on $A$ equals to $-\rho\circ R_{\nabla}^{\textrm{bas}}$, while the
curvature of the $A$-connection $\nabla^{\textrm{bas}}$ on $TM$
equals to $-R_{\nabla}^{\textrm{bas}}\circ \rho$.
\item $R^{\textrm{bas}}_{\nabla}$ is closed with respect to $\nabla^{\textrm{bas}}$ i.e.
$d_{\nabla^{\textrm{bas}}}(R^{\textrm{bas}}_{\nabla})= 0$.
\end{enumerate}
\end{proposition}

\begin{proof} For $\alpha, \beta, \gamma\in \Gamma(A)$,
\begin{eqnarray*}
R_{\nabla}^{\textrm{bas}}(\alpha,
\beta)\rho(\gamma)&=&\nabla_{\rho(\gamma)}([\alpha,\beta])
-[\nabla_{\rho(\gamma)}(\alpha),\beta]
-[\alpha,\nabla_{\rho(\gamma)}(\beta)]\\
&&-\nabla_{\nabla_{\beta}\rho(\gamma)}(\alpha)
+\nabla_{\nabla_{\alpha}\rho(\gamma)}(\beta)\\
&=&\nabla_{[\alpha,\beta]}(\gamma)-\nabla_{\alpha}(\nabla_{\beta}(\gamma))+\nabla_{\beta}(\nabla_{\alpha}(\gamma))\\
&=&-R_{\nabla^{\textrm{bas}}}(\alpha,\beta)(\gamma)
\end{eqnarray*}
On the other hand, if we evaluate at a vector field $X$ the
computation becomes:
\begin{eqnarray*}
\rho (R_{\nabla}^{\textrm{bas}}(\alpha, \beta)X)&=&
\rho(\nabla_X([\alpha,\beta])-[\nabla_X(\alpha),\beta]-[\alpha,\nabla_X(\beta)]\\
&&-\nabla_{\nabla_{\beta}X}(\alpha)
+\nabla_{\nabla_{\alpha}X}(\beta))\\
&=&\rho(\nabla_X([\alpha,\beta]))+[\rho([\alpha,\beta]),X]\\
&&+[[\rho(\beta),X],\rho(\alpha)]-\rho([\alpha,\nabla_X(\beta)])-\rho(\nabla_{\nabla_{\beta}X}(\alpha))\\
&&+[\rho(\beta),[\rho(\alpha),X]]+\rho(\nabla_{\nabla_{\alpha}X}(\beta))-\rho([\nabla_X(\alpha),\beta])\\
&=&\nabla_{[\alpha,\beta]}(X)- \nabla_{\alpha}(\nabla_{\beta}(X))+\nabla_{\beta}(\nabla_{\alpha}(X))\\
&=&-R_{\nabla^{\textrm{bas}}}(\alpha,\beta)(X)
\end{eqnarray*}
The proof of the second part is a similar computation that we will
omit.
\end{proof}

The following results indicate the geometric meaning of
the  basic curvature $R_{\nabla}^{\textrm{bas}}$. The first one
refers to the characterization of Lie algebroids which arise from
Lie algebra actions.

\begin{proposition}
A Lie algebroid $A$ over a simply connected manifold $M$ is the
algebroid associated to a Lie algebra action on $M$ if and only if it
admits a flat connection $\nabla$ whose induced  basic curvature
$R_{\nabla}^{\textrm{bas}}$ vanishes.
\end{proposition}

\begin{proof}
If $A$ is associated to a Lie algebra action, one chooses $\nabla$
to be the obvious flat connection. Assume now that there is a
connection $\nabla$ as above. Since $M$ is simply connected, the
bundle is trivial. Choose a frame of flat sections $\alpha_1, \dots,
\alpha_r$ of $A$, and write
\begin{equation*}
[\alpha_i,\alpha_j]=\sum_{k=1}^r c^k_{ij}\alpha_k
\end{equation*}
with $c_{ij}^{k}\in C^{\infty}(M)$. Since
\[ R_{\nabla}^{\textrm{bas}}(\alpha_i,\alpha_j)(X)=\sum_{k=1}^r \nabla_X(c^k_{ij}\alpha_k)=
\sum_{k=1}^r X(c^k_{ij})\alpha_k,\] we deduce that the
$c_{ij}^{k}$'s are constant. The Jacobi identity for the Lie bracket
on $\Gamma(A)$ implies that $c_{ij}^{k}$'s are the structure
constants of a Lie algebra, call it $\mathfrak{g}$. The anchor map
defines an action of $\mathfrak{g}$ on $M$, and the trivialization
of $A$ induces the desired isomorphism.
\end{proof}

In particular consider $A=TM$ for some compact simply connected
manifold $M$ and a flat connection $\nabla$ on $A$. The condition
that the basic curvature vanishes means precisely that the
conjugated connection $\overline{\nabla}$ defined by:
\begin{equation*}
\overline{\nabla}_X(Y)=[X,Y]+\nabla_Y(X)
\end{equation*}
is also flat. This implies that $M$ is a Lie group.

The next result refers to the relation between bundles of Lie
algebras (viewed as Lie algebroids with zero anchor map) and Lie
algebra bundles, for which the fiber Lie algebra is fixed in the
local trivializations.

\begin{proposition}
Let $A$ be a bundle of Lie algebras over $M$. Then $A$ is a  Lie
algebra bundle if and only if it admits a connection $\nabla$ whose
basic curvature  vanishes.
\end{proposition}
\begin{proof}
If the bundle of Lie algebras is locally trivial then locally one
can choose connections with zero $A$-curvature. Then one can use
partitions of unity to construct a global connection with the same
property. For the converse, assume there exists a $\nabla$ with
$R_{\nabla}^{A}= 0$, and we need to prove that the Lie algebra
structure on the fiber is locally trivial. We may assume that $A=R^n
\times R^r$ as a vector bundle. The vanishing of the basic curvature
means that $\nabla$ acts as derivations of the Lie algebra fibers:
\begin{equation*}
\nabla_X([\alpha,\beta])=[\nabla_X(\alpha),\beta]+[\alpha,\nabla_X(\beta)]
\end{equation*}
Since derivations are infinitesimal automorphisms, we deduce that
the parallel transports induced by $\nabla$ are Lie algebra
isomorphisms, providing the necessary Lie algebra bundle
trivialization.
\end{proof}

\section{Representations up to homotopy}
\label{Representations up to homotopy}

\subsection{Representations up to homotopy and first examples}
\label{Representations up to homotopy and first examples}

In this section we introduce the notion of representation up to
homotopy and the adjoint representation of Lie algebroids. As
before, $A$ is a Lie algebroid over $M$. We start with the shortest,
but less intuitive description of representations up to homotopy.

\begin{definition} A representation up to homotopy of $A$ consists of a graded vector bundle $E$ over $M$
and an operator, called the structure operator,
\begin{equation*}
D:\Omega(A;E)\rightarrow \Omega(A;E)
\end{equation*}
which increases the total degree by one and satisfies $D^2= 0$ and
the graded derivation rule:
\begin{equation*}
D(\omega \eta)=d_A(\omega) \eta +(-1)^k \omega D(\eta)
\end{equation*}
for all $\omega\in \Omega^k(A)$, $\eta\in \Omega(A; E)$.
The cohomology of the resulting complex is denoted by
$H^{\bullet}(A; E)$.
\end{definition}

Intuitively, a representation up to homotopy of $A$ is a complex
endowed with an $A$-connection which is ``flat up to homotopy''. We
will make this precise in what follows.

\begin{proposition}\label{Rep-long}
There is a 1-1 correspondence between representations up to homotopy
$(E, D)$ of $A$ and graded vector bundles $E$ over $M$ endowed with
\begin{enumerate}
\item A degree $1$ operator $\partial$ on $E$ making $(E, \partial)$ a complex.
\item An $A$-connection $\nabla$ on $(E, \partial)$.
\item An $\textrm{End(E)}$-valued $2$-form $\omega_{2}$ of total degree 1, i.e.
\[ \omega_{2}\in \Omega^2(A; \underline{\textrm{End}}^{-1}(E))\]
satisfying
\[ \partial(\omega_{[2]})+ R_{\nabla}= 0,\]
where $R_{\nabla}$ is the curvature of $\nabla$.
\item For each $i> 2$ an $\textrm{End(E)}$-valued $i$-form $\omega_{i}$ of total degree 1, i.e.
\[ \omega_{i}\in \Omega^i(A; \underline{\textrm{End}}^{1-i}(E))\]
satisfying
\[ \partial(\omega_{i})+ d_{\nabla}(\omega_{i-1})+ \omega_{2}\circ \omega_{i-2}+
\omega_{3}\circ \omega_{i-3}+ \ldots + \omega_{i-2}\circ \omega_{2}=
0.\]
\end{enumerate}
The correspondence is characterized by
\[ D(\eta)= \partial(\eta)+ d_{\nabla}(\eta)+ \omega_{2}\wedge \eta+ \omega_{3}\wedge \eta+ \ldots .\]
\end{proposition}

We also write
\begin{equation}
\label{structure} D= \partial+ \nabla + \omega_{2}+ \omega_{3}+
\ldots .
\end{equation}

\begin{proof}
Due to the derivation rule and the fact that $\Omega(A; E)$ is
generated as an $\Omega(A)$-module by $\Gamma(E)$, the operator $D$
will be uniquely determined by what it does on $\Gamma(E)$. It will
send each $\Gamma(E^k)$ into the sum
\[ \Gamma(E^{k+1})\oplus \Omega^1(A; E^k)\oplus \Omega^2(A; E^{k-1})\oplus \ldots ,\]
hence it will also send each $\Omega^p(A; E^k)$ into the sum
\[ \Omega^p(A; E^{k+1})\oplus \Omega^{p+1}(A; E^k)\oplus \Omega^{p+2}(A; E^{k-1})\oplus \ldots ,\]
and we denote by $D_0$, $D_1$, $\ldots$ the components of $D$. From
the derivation rule for $D$, we deduce that each $D_i$ for $i\neq 1$
is a (graded) $\Omega(A)$-linear map and, by Lemma
\ref{lemma-graded}, it is the wedge product with an element in
$\Omega(A; \underline{\textrm{End}}(E))$. On the other hand, $D_1$
satisfies the derivation rule on each of the vector bundles $E^k$
and, by Proposition \ref{connections-derivations}, it comes from
$A$-connections on these bundles. The equations in the statement
correspond to $D^2= 0$.
\end{proof}

Next, one can define the notion of morphism between representations
up to homotopy.

\begin{definition} A morphism $\Phi: E\to F$ between two representations up to homotopy
of $A$ is a degree zero linear map
\begin{equation*}
\Phi:\Omega(A;E) \rightarrow \Omega(A;F)
\end{equation*}
which is $\Omega(A)$-linear and commutes with the structure
differentials $D_E$ and $D_F$.

We denote by $\RRep^{\infty}(A)$ the resulting category, and by
$\mathrm{Rep}^{\infty}(A)$ the set of isomorphism classes of representations
up to homotopy of $A$.
\end{definition}

By the same arguments as above, one gets the following description
of morphisms in $\RRep^{\infty}(A)$. A morphism is necessarily of
type
\begin{equation*}
\Phi=\Phi_0 +\Phi_1+\Phi_2 + \cdots
\end{equation*}
where $\Phi_i$ is a $\textrm{Hom}(E, F)$-valued $i$-form on $A$ of
total degree zero:
\[ \Phi_i \in \Omega^i(A; \textrm{Hom}^{-i}(E,F))\]
satisfying
\[ \partial(\Phi_n)+ d_{\nabla}(\Phi^{n-1})+ \sum_{i+j= n, i\geq 2} [\omega_{i}, \Phi_{j}]= 0.\]
Note that, in particular, $\Phi_0$ must be a map of complexes.

\begin{example}[Usual representations]\rm \ Of course, any representation $E$ of $A$ can be seen as a representation
up to homotopy concentrated in degree zero. More generally, for any
integer $k$, one can form the representation up to homotopy $E[-k]$,
which is $E$ concentrated in degree $k$.
\end{example}

\begin{example}[Differential forms]\label{Differential forms}\rm \  Any closed form $\omega\in \Omega^n(A)$ induces a representation up
to homotopy on the complex which is the trivial line bundle in
degrees $0$ and $n-1$, and zero otherwise. The structure operator is
$\nabla^{\textrm{flat}}+ \omega$ where $\nabla^{\textrm{flat}}$ is
the flat connection on the trivial line bundle. If $\omega$ and
$\omega'$ are cohomologous, then the resulting representations up to
homotopy are isomorphic with isomorphism defined by $\Phi_0=
\textrm{Id}$, $\Phi_{n-1}= \theta\in \Omega^{n-1}(A)$ chosen so that
$d(\theta)= \omega- \omega'$. In conclusion, there is a well defined
map $H^{\bullet}(A)\to \textrm{Rep}^{\infty}(A)$.
\end{example}

\begin{example}[Conjugation]\rm \ For any representation up to homotopy $E$ with
structure operator $D$ given by (\ref{structure}), one can form a
new representation up to homotopy $\overline{E}$, which has the same
underlying graded vector bundle as $E$, but with the structure
operator
\[ \overline{D}= - \partial + \nabla - \omega_2+ \omega_3- \omega_4 + \ldots .\]
In general, $E$ and $\overline{E}$ are isomorphic, with isomorphism
$\Phi= \Phi_0$ equal to $(-1)^n\textrm{Id}$ on $E^n$.
\end{example}

\begin{remark}\label{rk-length-1}\rm\ Let us now be more explicit on the building blocks of
representations up to homotopy which are concentrated in two
consecutive degrees, say $0$ and $1$. From Proposition
\ref{Rep-long} we see that such a representation consists of
\begin{enumerate}
\item Two vector bundles $E$ and $F$, and a vector bundle map $f: E\to F$.
\item $A$-connections on $E$ and $F$, both denoted $\nabla$, compatible with $\partial$ ($\nabla_{\alpha} \partial= \partial \nabla_{\alpha}$).
\item A 2-form $K\in \Omega^2(A; \textrm{Hom}(F, E))$ such that
\[ R_{\nabla^E}= \partial\circ K,\ R_{\nabla^F}= K\circ \partial\]
and such that $d_{\nabla}(K)= 0$.
\end{enumerate}
\end{remark}

\begin{example} [The double of a vector bundle]\label{The double of a vector bundle}\rm \  Let
$E$ be a vector bundle over $M$. For any $A$-connection $\nabla$ on
$E$ with curvature $R_{\nabla}\in \Omega^2(A;
\underline{\textrm{End}}(E))$, the complex
$E\stackrel{\textrm{Id}}{\to} E$ concentrated in degrees $0$ and
$1$, together with the structure operator
\[ D_{\nabla}:= \textrm{Id}+ \partial+ R_{\nabla}\]
defines a representation up to homotopy of $A$ denoted
$\mathcal{D}_{E, \nabla}$. The resulting element
\[ \mathcal{D}_E\in \textrm{Rep}^{\infty}(A)\]
does not depend on the choice of the connection. To see this, remark
that if $\nabla'$ is another $A$-connection, then there is an
isomorphism
\[ \Phi: (\mathcal{D}_E, D_{\nabla})\to (\mathcal{D}_E, D_{\nabla'})\]
with two components:
\[ \Phi^0= \textrm{Id}, \Phi^{1}(\alpha)= \nabla_{\alpha}- \nabla_{\alpha}' .\]
\end{example}

We will now explain how representations up to homotopy of length $1$
are related to extensions.

\begin{proposition}\label{representations and extensions} For any representation up
to homotopy of length one with vector bundles $E$ in degree $0$ and
$F$ in degree $1$ and structure operator $D= \partial+ \nabla+ K$,
there is an extension of Lie algebroids:
\[ \mathfrak{g}_{\partial}\to \tilde{A}\to A\]
where
\begin{enumerate}
\item $\mathfrak{g}_{\partial}= \textrm{Hom}(F, E)$, is a bundle of Lie algebras with bracket
$[S, T]_{\partial}= S\partial T- T\partial S$.
\item $\tilde{A}= \mathfrak{g}_{\partial}\oplus A$ with anchor $(S, \alpha)\mapsto \rho(\alpha)$
and bracket
\[ [(S, \alpha), (T, \beta)]=  ([S,T]+\nabla_{\alpha}(T)-\nabla_{\beta}(S)+
K(\alpha,\beta),[\alpha,\beta]).\]
\end{enumerate}

\end{proposition}

\begin{proof} After a careful computation, we find that the Jacobi
identity for the bracket of $\tilde{A}$ breaks into the following
equations (cf. Theorem 7.3.7 in \cite{McK}):
\begin{equation}\label{ext1}
\nabla_{\alpha}([S,
T])=[\nabla_{\alpha}(S),T]+[S,\nabla_{\alpha}(T)]
\end{equation}
\begin{equation}\label{ext2}
\nabla_{[\beta,\gamma]}(T)-\nabla_{\beta}\nabla_{\gamma}(T)+\nabla_{\gamma}\nabla_{\beta}(T)=[T,K(\beta,\gamma)]
\end{equation}
\begin{equation}\label{ext3}
K([\alpha,\beta],\gamma)+K([\beta,\gamma],\alpha)+K([\gamma,\alpha],\beta)=\nabla_{\beta}(K(\gamma,\alpha))+
\nabla_{\alpha}(K(\beta,\gamma))+\nabla_{\gamma}(K(\alpha,\beta))
\end{equation}
for $\alpha, \beta, \gamma \in \Gamma(A)$ and $S, T\in
\Gamma(\frak{g}_{\partial})$. These equations are not equivalent to,
but they follow from the equations satisfied by $\partial$, $\nabla$
and $K$. The first equation follows from the compatibility of
$\partial$ and $\nabla$, the second equation follows from the two
equations satisfied by the curvature, while the last equation is
precisely $d_{\nabla}(K)= 0$.
\end{proof}

\begin{example}\label{Atiyah-p} \rm When $A= TM$ and $E$ is a vector bundle, the extension associated to
the double of $E$ (Example \ref{The double of a vector bundle})
is isomorphic to the ``Atiyah extension'' induced by $E$:
\[ \textrm{End}(E)\to \mathfrak{gl}(E)\to TM.\]
This extension is discussed e.g. in Section 1 of \cite{MY} (where
$\mathfrak{gl}(E)$ is denoted by $\mathcal{D}(E)$). Recall that
$\mathfrak{gl}(E)$ is the vector bundle over $M$ whose sections are
the derivations of $E$, i.e. pairs $(D, X)$ consisting of a linear
map $D: \Gamma(E)\to \Gamma(E)$ and a vector field $X$ on $M$, such
that $D(fs)= fD(s)+ L_{X}(f)s$ for all $f\in C^{\infty}(M)$, $s\in
\Gamma(E)$. The Lie bracket of $\mathfrak{gl}(E)$ is the commutator
\[ [(D, X), (D', X')]= (D\circ D'- D'\circ D, [X, X']),\]
while the anchor sends $(D, X)$ to $X$. A connection on $E$ is the
same thing as a splitting of the Atiyah extension, and it induces an
identification
\[ \mathfrak{gl}(E)\cong \textrm{End}(E)\oplus TM.\]
Computing the bracket (or consulting \cite{MY}) we find that, after
this identification, the Atiyah extension becomes the extension
associated to the double $\mathcal{D}_E$.
\end{example}

\subsection{The adjoint representation}
\label{subsection: The adjoint representation}

It is now clear that the properties
of the basic connections and the basic curvature given in
Proposition \ref{properties-basic} give the adjoint complex the
structure of a representation up to homotopy. Choosing a connection
on the vector bundle $A$, the adjoint complex
\[ A\stackrel{\rho}{\to} TM\]
together with the structure operator
\[ D_{\nabla}:= \rho + \nabla^{\textrm{bas}}+ R^{\textrm{bas}}_{\nabla}\]
becomes a representation up to homotopy of $A$, denoted
$\mathrm{ad}_{\nabla}$. The isomorphism class of this representation is
called the adjoint representation of $A$ and is denoted
\[ \textrm{ad}\in \textrm{Rep}^{\infty}(A) .\]

\begin{theorem}\label{The adjoint representation}
Given two connections on $A$, the corresponding adjoint representations
are naturally isomorphic. Also, there is an isomorphism:
\[ H^{\bullet}(A; \mathrm{ad}) \cong H^{\bullet}_{\mathrm{def}} (A).\]
\end{theorem}

\begin{proof} Let $\nabla$ and $\nabla'$ be two connections on $A$. Then,
$\Phi=\Phi_0+\Phi_1$ where
\[ \Phi_0= \textrm{Id} \text{,  }  \Phi_1(\alpha)(X)= \nabla_{X}(\alpha)- \nabla_{X}'(\alpha),\]
defines an isomorphism between the corresponding adjoint representations. For
the last part, note that there is an exact sequence
\[ 0\to \Omega^k(A; A)\to C^{k}_{\textrm{def}}(A)\stackrel{-\sigma}{\to} \Omega^{k-1}(A; TM)\to 0,\]
where $\sigma$ is the symbol map (see Subsection
\ref{subsection:deformations}). A connection
$\nabla$ on $A$ induces a splitting of this sequence, and then an
isomorphism
\begin{eqnarray*}
\Psi: C_{\textrm{def}}^k(A) &\rightarrow &  \Omega^k(A; A)\oplus\Omega^{k-1}(A; TM)= \Omega(A, \textrm{ad})^k\\
D & \mapsto & (c_D,-\sigma_D)
\end{eqnarray*}
where $\sigma_D$ is the symbol of $D$ and $c_D$ is given by
\[ c_D(\alpha_1, \dots,\alpha_k)= D(\alpha_1, \dots ,\alpha_k)+(-1)^{k-1}
\sum_{i=1}^k(-1)^{i} \nabla_{ \sigma(D)(\alpha_1, \dots
,\widehat{\alpha_i}, \dots,\alpha_k)}(\alpha_i). \] This map is an
isomorphism of $\Omega(A)$-modules. We need to prove that the
operators $\delta$ and $D_{\nabla}$ coincide. Since these two
operators are derivations with respect to the module structures, it
is enough to prove that they coincide in low degrees, and this can be checked by inspection.
\end{proof}

We will now explain the relationship between the adjoint representation and the first jet algebroid.  
 We denote by $J^1(A)$ the first jet bundle of $A$
and by $\pi: J^1(A)\to A$ the canonical projection. For a section $\alpha$ of $A$ we denote by $j^1(\alpha)\in \Gamma(J^1A)$ its
first jet. It is well-known (for a proof, see e.g. \cite{CF2}) that
$J^1(A)$ admits a unique Lie algebroid structure such that for any
section $\alpha$ of $A$,
\begin{equation}\label{jet1}
\rho(j^1\alpha)=\rho(\alpha)
\end{equation}
and for sections $\alpha, \beta \in \Gamma(A)$ :
\begin{equation}\label{jet2}
[j^1\alpha, j^1\beta]=j^1([\alpha, \beta]).
\end{equation}
The jet Lie algebroid fits into a short exact sequence of Lie algebroids
\begin{equation}\label{jet-extension}
\textrm{Hom}(TM, A)\stackrel{i}{\to} J^1(A) \stackrel{\pi}{\to} A ,
\end{equation}
where the inclusion $i$ is determined by the condition
\[ \textrm{Hom}(TM, A)\ni df\otimes \alpha \mapsto fj^{1}(\alpha)- j^1(f\alpha),\]
for all $f\in C^{\infty}(M)$, $\alpha\in \Gamma(A)$.
On the other hand, Proposition \ref{representations and extensions}
associates to the adjoint representation an extension of Lie algebroids.

\begin{proposition} \label{cor-first-jet} The extension associated to the adjoint representation 
is isomorphic to the first-jet extension $\textrm{Hom}(TM, A)\stackrel{i}{\to} J^1(A) \stackrel{\pi}{\to} A$.
\end{proposition} 

\begin{proof} Note that, taking global sections in (\ref{jet-extension}), we obtain a split short exact sequence with splitting $j^1$. With respect to the resulting decomposition, the Lie bracket on $\Gamma(J^1A)$ is 
given by (\ref{jet2}), the Lie bracket on $Hom(TM, A)$:
\[ [T, S]= T\circ \rho\circ S- S\circ \rho\circ T ,\]
and the Lie bracket $[j^1(\alpha), T]\in Hom(TM, A)$ between $j^1(\alpha)$ and $T\in Hom(TM, A)$:
\[ [j^1(\alpha), T](X)= T([X, \rho(\alpha)])+ [\alpha, T(X)].\]
The last two formulas follow from the Leibniz identity for the bracket on $\Gamma(J^1(A))$ and writing
$T$ and $S$ as a sum of expressions of type $df\otimes \alpha= fj^{1}(\alpha)- j^1(f\alpha)$. 

Next, giving a connection $\nabla$ on $A$ is equivalent to choosing a vector bundle
splitting $j^{\nabla}: A\to J^1(A)$ of $\pi$. The relation between
the two is given by
\[ j^{\nabla}(\alpha)= j^{1}(\alpha)+ \nabla_{\cdot}(\alpha),\]
for all $\alpha\in \Gamma(A)$. Here $\nabla_{\cdot}(\alpha)\in
\textrm{Hom}(TM, A)$ is given by $X\mapsto \nabla_{X}(\alpha)$. Using this and the previous formulas,
a careful but straightforward computation shows that 
\[ [j^{\nabla}(\alpha), T]= \nabla^{\textrm{bas}}_{\alpha}\circ T- T\circ \nabla^{\textrm{bas}}_{\alpha}= \nabla^{\textrm{bas}}_{\alpha}(T),\]
\[ j^{\nabla}([\alpha, \beta])- [j^{\nabla}(\alpha), j^{\nabla}(\beta)]= R_{\nabla}^{\textrm{bas}}(\alpha, \beta),\]
for all $\alpha, \beta\in \Gamma(A)$, $T\in \Gamma(Hom(TM, A))$. This shows that the vector bundle isomorphism
$J^1(A)\cong A\oplus Hom(TM, A)$ induced by the splitting $j^{\nabla}$ identifies the Lie algebroid bracket of $J^1(A)$ with the
extension bracket of Proposition \ref{representations and extensions} applied to $\textrm{ad}_{\nabla}$.
\end{proof}

We note that there doesn't seem to be any construction which
associates to a Lie algebroid extension of $A$ a representation up
to homotopy so that, applying it to $J^1(A)$ one  recovers the
adjoint representation. In other words, $J^1(A)$ with its structure
of extension of $A$ does not contain all the information about the
structure of the adjoint representation.

\subsection{The case of tangent bundles}

The representations up to homotopy of $TM$ are connections on
complexes of vector bundles which are \textit{flat up to homotopy}.
Indeed, at least the first equation in Proposition \ref{Rep-long}
says that the curvature of $\nabla$ is trivial cohomologically (up
to homotopy).

On the other hand, a flat connection $\nabla$ on a vector bundle $E$
can be integrated to a representation of the fundamental groupoid
of $M$. This correspondence is induced by parallel transport. To be
more precise, given a vector bundle $E$ endowed with a connection
$\nabla$, for any path $\gamma$ in $M$ from $x$ to $y$, the parallel
transport along $\gamma$ with respect to $\nabla$ induces a linear
isomorphism
\[ T_{\gamma}: E_{x}\to E_{y} .\]
This construction is compatible with path concatenation. When
$\nabla$ is flat, $T_{\gamma}$ only depends on the homotopy class of
$\gamma$, and this defines an action of the homotopy groupoid of
$M$ on $E$. It is only natural to ask what is the corresponding notion
of parallel transport for connections which are flat up to homotopy.
\begin{proposition}\label{parallel} Let $(E, D)$ be a representation up to homotopy of $TM$.
Then:
\begin{enumerate}
\item For any path $\gamma$ in $M$ from $x$ to $y$, there is an induced
chain map
\[ T_{\gamma}: (E_x, \partial)\to (E_y, \partial)\]
and this construction is compatible with path concatenation. More
precisely, $T_{\gamma}$ is the parallel transport with respect to
the connection underlying $D$.
\item If $\gamma_0$ and $\gamma_1$ are two homotopic paths in $M$ from $x$ to $y$
then $T_{\gamma_0}$ and $T_{\gamma_1}$ are chain homotopic. More
precisely, for any homotopy $h$ between $\gamma_0$ and $\gamma_1$
there is an associated map of degree $-1$, $T_h: E_x\to E_y$, such
that
\[ T_{\gamma_1}- T_{\gamma_0}= [\partial, T_{h}].\]
\end{enumerate}
\end{proposition}

\begin{proof} The compatibility of $\nabla$ with the grading and $\partial$ implies that
the parallel transport $T_{\gamma}$ is a map of chain complexes. We
now prove (2). Given a path $u: I\to E$ ($I= [0, 1]$), sitting over
some base path $\gamma: I\to M$, we denote by
\[ \frac{Du}{Dt}= \nabla_{\frac{d\gamma}{dt}}(u) \]
the derivative of $u$ with respect to the connection $\nabla$. Then,
for any path $\gamma$, and any $s, t$, the parallel transport
\[ T_{\gamma}^{s, t}: E_{\gamma(s)}\to E_{\gamma(t)} \]
is defined by the equation
\[ \frac{D}{Dt} T_{\gamma}^{s, t}(u)= 0,\ \ T_{\gamma}^{s, s}(u)= u.\]
The global parallel transport along $\gamma$, $T_{\gamma}: E_x\to
E_{y}$ is obtained for $s= 0$, $t= 1$. Note that, for a path in the
fiber above $\gamma(s)$,  $\phi: I\to E_{\gamma(s)}$, one has
\[ \frac{D}{Dt} T_{\gamma}^{s, t}(\phi(t))= T_{\gamma}^{s, t}(\frac{d\phi}{dt}(t)).\]
This implies that, for a path $v: I\to E$ above $\gamma$ and $u_0\in
E_x$, the unique solution of the equation
\[ \frac{Du}{dt}= v, \ u(0)= u_0\]
can be written in terms of the parallel transport as
\[ u(t)= T_{\gamma}^{0, t} (u_0+ \int_{0}^{t} T_{\gamma}^{t', 0}(v(t'))dt').\]

For any map $u: I\times I\to E$ sitting above some $h: I\times I\to
M$, we have
\[ \frac{D^2u}{DtD\epsilon}- \frac{D^2u}{D\epsilon Dt}= R(\frac{dh}{dt}, \frac{dh}{d\epsilon})
(u(\epsilon, t)),\]
where $R= R_{\nabla}$ is the curvature of $\nabla$. Let now $h$ be
as in the statement, $u_0\in E_x$, and consider in the previous
equation applied to
\[ u(\epsilon, t)= T_{\gamma_{\epsilon}}^{0, t}(u_0) .\]
We find
\[ \frac{D}{Dt} (\frac{Du}{D\epsilon})= R(\frac{dh}{dt}, \frac{dh}{d\epsilon})u,\]
where $\gamma_{\epsilon}= h(\epsilon, \cdot)$. Since
$\frac{Du}{D\epsilon}(\epsilon, 0)= 0$, we find
\[ \frac{Du}{D\epsilon}= T_{\gamma_{\epsilon}}^{0, t}\int_{0}^{t} T_{\gamma_{\epsilon}}^{t', 0}
R(\frac{dh}{dt}, \frac{dh}{d\epsilon})u(\epsilon, t') dt'.\] Fixing
the argument $t$, since
\[ u(0, t)= T_{\gamma_0}^{0, t}(u_0),\]
by the same argument as above, we deduce that
\[ u(\epsilon, t)= T_{h_t}^{0, \epsilon}[T_{\gamma_0}^{0, t}(u_0)+ \int_{0}^{\epsilon} T_{h_t}^{\epsilon', 0}T_{\gamma_{\epsilon'}}^{0, t}\int_{0}^{t}T_{\gamma_{\epsilon'}}^{t', 0} R(\frac{dh}{dt}, \frac{dh}{d\epsilon})u(\epsilon', t') dt' d\epsilon'] ,\]
where $h_t(\cdot)= h(\cdot, t)$. Taking $\epsilon= 1$, $t= 1$, we
find
\[ T_{\gamma_1}(u_0)= T_{\gamma_0}(u_0)+ \int_{0}^{1}\int_{0}^{1} T_{\gamma_{\epsilon}}^{t, 1}R(\frac{dh}{dt}, \frac{dh}{d\epsilon})T_{\gamma_{\epsilon}}^{0, t}(u_0)dtd\epsilon .\]
Using now that $R+ \partial(\omega_2)= 0$, we deduce that
\[ T_{\gamma_1}(u_0)- T_{\gamma_0}(u_0)= [\partial, T_{h}]u_0,\]
where $T_h\in \textrm{Hom}(E_x, E_y)$ is
\[ T_{h}= -\int_{0}^{1} \int_{0}^{1} T_{\gamma_{\epsilon}}^{t, 1}\omega_2(\frac{dh}{dt}, \frac{dh}{d\epsilon})T_{\gamma_{\epsilon}}^{0, t}dtd\epsilon . \]
Note that the expression under the integral is in the $(\epsilon,
t)$-independent vector space $\textrm{Hom}(E_x, E_y)$.
\end{proof}

We would like to mention here the interesting recent work of Igusa \cite{Igusa} where, based on Chen's iterated integrals \cite{Chen}, the author 
constructs a general parallel transport for flat superconnections.

\section{Operations, cohomology and the derived category}
\label{Operations, cohomology and the derived category}

\subsection{Operations and more examples}
\label{Operations, and more examples}

As explained in the appendix, the standard operations on vector
spaces such as
\[ E\mapsto E^*, E\mapsto \Lambda(E), E\mapsto S(E),\]
\[ (E, F)\mapsto E\oplus F, (E, F)\mapsto E\otimes F, (E, F)\mapsto \underline{\textrm{Hom}}(E, F)\]
extend to the setting of graded vector bundles, complexes of vector
bundles and complexes of vector bundles endowed with a connection.
We will see that these operations are also well defined for representations up to homotopy.

\begin{example}[Taking duals]\rm \  For $E\in \RRep^{\infty}(A)$ with associated structure operator $D$,
the operator $D^*$ corresponding to the dual $E^*$ is uniquely
determined by the condition
\[ d_A(\eta\wedge \eta')= D^*(\eta)\wedge \eta'+ (-1)^{|\eta|} \eta\wedge D(\eta'),\]
for all $\eta\in \Omega(A; E^*)$ and $\eta'\in \Omega(A; E)$, where
 $\wedge$ is the operation
\[ \Omega(A; E^*)\otimes \Omega(A; E)\to \Omega(A)\]
induced by the pairing between $E^*$ and $E$ (see Appendix
\ref{The graded setting}). In terms of the components of $D$, if $D=
\partial + \nabla+ \sum_{i\geq 2} \omega_i$, we find $D^*=
\partial^{*} + \nabla^{*}+ \sum_{i\geq 2} \omega_{i}^{*}$, where
$\nabla^*$ is the connection dual to $\nabla$ and, for $\eta_k\in
(E^k)^*$,
\[ \partial^*(\eta)= -(-1)^k\eta\circ \partial,\ \  \omega_{p}^{*}(\alpha_1, \ldots , \alpha_p)(\eta_k)= - (-1)^{k(p+1)} \eta_k\circ \omega_{p}(\alpha_1, \ldots , \alpha_p).\]
In particular, if we start with a representation up to homotopy of
length one, $D= \partial + \nabla+ K$ on $E\stackrel{\partial}{\to}
F$ ($E$ in degree $0$ and $F$ in degree $1$), the dual complex will
be $F^*\stackrel{\partial^*}{\to} E^*$ ($F^*$ in degree $-1$ and
$E^*$ in degree $0$), with $D^*=
\partial^*+ \nabla^*- K^*$. The fact that some signs appear when
taking duals is to be expected since, for any connection $\nabla$,
the curvature of $\nabla^*$ equals the negative of the dual of the
curvature of $\nabla$.
\end{example}

\begin{example}[Tensor products]\rm \  For $E, F\in \RRep^{\infty}(A)$, with associated structure operators
$D^{E}$ and $D^{F}$, the operator $D$ corresponding to $E\otimes F$
is uniquely determined by the condition
\[ D(\eta_1\wedge \eta_2)= D^{E}(\eta_1)\wedge \eta_2+ (-1)^{|\eta_1|} \eta_1\wedge D^{F}(\eta_2),\]
for all $\eta_1\in \Omega(A; E)$ and $\eta_2\in \Omega(A; F)$. More
explicitly, if $D^E= \partial^{E}+ \nabla^{E}+ \omega_{2}^{E}+
\cdots$ and similarly for $D^F$, then $D= \partial+ \nabla+
\omega_2+ \cdots$, where:
\begin{enumerate}
\item $\partial$ is the graded tensor product of $\partial^E$ and $\partial^F$:
$\partial= \partial^E\otimes \textrm{Id}+ \textrm{Id}\otimes
\partial^F$,
\[ \partial (u\otimes v)= \partial^E(u)\otimes v+ (-1)^{|u|}u\otimes \partial^F(v).\]
\item $\nabla$ is just the tensor product connection of $\nabla^E$ and $\nabla^F$: $d_{\nabla}=
d_{\nabla^E}\otimes \textrm{Id}+ \textrm{Id}\otimes d_{\nabla^F}$,
\[ \nabla_{\alpha}(u\otimes v)= \nabla_{\alpha}^E(u)\otimes v+ u\otimes \nabla_{\alpha}^F(v).\]
\item $\omega_p= \omega_{p}^{E}\otimes \textrm{Id}+ \textrm{Id}\otimes \omega_{p}^F$.
\end{enumerate}
\end{example}

\begin{example}[Pull-back]\rm \
A Lie algebroid $A$ over $M$ can be pulled-back along
a submersion $\tau: N\to M$ or, more generally, along smooth maps
$\tau$ which satisfy certain transversality condition, as we now explain. Recall \cite{MH}
that the pull-back algebroid $\tau^{!}A$ has the fiber at $x\in N$:
\[ \tau^{!}(A)_x= \{ (X, \alpha): X\in T_xN, \alpha\in A_{\tau(x)}, (d\tau)(X)= \rho(\alpha)\}.\]
The transversality condition mentioned above is that this is a smooth vector bundle
over $N$, which certainly happens if $\tau$ is a submersion or the
inclusion of a leaf of $A$. The anchor of $\tau^!A$ sends $(X,
\alpha)$ to $X$, while the bracket is uniquely determined by the
derivation rule and
\[ [(X, \tau^{*}\alpha), (Y, \tau^{*}\beta)]= ([X, Y], \tau^{*}[\alpha, \beta]) .\]
In general, there is a pull-back map (functor)
\[ \tau^{*}: \RRep^{\infty}(A)\to \RRep^{\infty}(\tau^{!}(A))\]
which sends $E$ with structure operator $D= \partial+ \nabla+ \sum
\omega_i$ to $\tau^*(E)$ endowed with $D= \partial+ \tau^*(\nabla)+
\sum \tau^{*}(\omega_i)$ where $\tau^*\nabla$ is the pull-back
connection
\[ (\tau^*\nabla)_{(X, \alpha)}(\tau^*(s))= \tau^* (\nabla_{\alpha}(s)),\]
while
\[ \tau^{*}(\omega_i)((X_1, \alpha_1), \ldots , (X_i, \alpha_i))= \omega_i(\alpha_1, \ldots , \alpha_i).\]
\end{example}

\begin{example}[Semidirect products with representations up to homotopy] \rm \
If $\mathfrak{g}$ is a Lie algebra and $V\in
\RRep^{\infty}(\mathfrak{g})$, the operator making $\Lambda(V^*)$ a
representation up to homotopy of $\mathfrak{g}$ is a derivation on
the algebra
\[ \Lambda(\mathfrak{g}^*)\otimes \Lambda(V^*)= \Lambda((\mathfrak{g}\oplus V)^*),\]
i.e. defines the structure of $L_{\infty}$-algebra (see \cite{Sta})
on the direct sum $\mathfrak{g}\oplus V$. This $L_{\infty}$-algebra
deserves the name ``semi-direct product of $\mathfrak{g}$ and
$V$'', and is denoted $\mathfrak{g}\ltimes V$ (it is the usual
semi-direct product if $V$ is just a usual representation).
\end{example}

\begin{example}[Exterior powers of the adjoint representation and $k$-differentials]\rm \ \label{exterior}
In the literature one often encounters equations which look like cocycle conditions, but which do not seem to have a cohomology theory behind them. Such equations arise
naturally as infinitesimal
manifestations of properties of global objects and it
is often useful to interpret them as part of cohomology theories.
We now point out one such example.

An almost $k$-differential (\cite{IG}) on a Lie algebroid $A$ is a
pair of linear maps $\delta: C^{\infty}(M)\to
\Gamma(\Lambda^{k-1}A)$, $\delta: \Gamma(A)\to \Gamma(\Lambda^kA)$
satisfying
\begin{enumerate}
\item[(i)] $\delta(fg)= \delta(f)g+ f\delta(g),$
\item[(ii)] $\delta(f\alpha)= \delta(f)\wedge \delta(\alpha)+ f\delta(\alpha),$
\end{enumerate}
for all $f, g\in C^{\infty}(M)$, $\alpha\in \Gamma(A)$. It is called
a $k$-differential if
\[ \delta [\alpha, \beta]= [\delta(\alpha), \beta] + [\alpha, \delta(\beta)]\]
for all $\alpha, \beta\in \Gamma(A)$.


We will now explain how $k$-differentials are related to representations up to homotopy.
Applying the exterior powers construction to the adjoint
representation, we find new elements:
\[ \Lambda^k\textrm{ad}\in \textrm{Rep}^{\infty}(A),\]
one for each positive integer $k$. These are given by the
representations up to homotopy $\Lambda^k(\textrm{ad}_{\nabla})$,
where $\nabla$ is an arbitrary connection on $A$. Generalizing the
case of the cohomology of $A$ with coefficients in the adjoint
representation $\textrm{ad}$, we now show that the cohomology with
coefficients in $\Lambda^k \textrm{ad}$ can be computed by a
complex which does not require the use of a connection.
More precisely, we define $(C^{\bullet}(A; \Lambda^k\textrm{ad}),
d)$ as follows. An element $c\in C^p(A; \Lambda^k\textrm{ad})$ is a
sequence $c= (c_0, c_1, \ldots )$ where
\[ c_{i}: \underbrace{\Gamma(A)\times \ldots \times \Gamma(A)}_{(p-i)\ \textrm{times}}\to \Gamma(\Lambda^{k- i}A\otimes S^{i}TM),\]
are multilinear, antisymmetric maps related by:
\[ c_i(\alpha_1, \ldots , f\alpha_{p-i})= fc_i(\alpha_1, \ldots , f\alpha_{p-i})+ i(df)(c_{i+1}(\alpha_1, \ldots, \alpha_{p-i-1})\wedge\alpha_{p-i}))\]
where $i(df): S^{i+1}(TM)\to S^i(TM)$ is the contraction by $df$. We
think of $c_1, c_2, \ldots$ as \textit{the tail} of $c_0$, which
measures the failure of $c_0$ to be $C^{\infty}(M)$-linear. For
instance, to define the differential $dc$, we first define its leading term by the Koszul formula
\begin{eqnarray*}
(dc)_0(\alpha_1, \dots ,\alpha_{p+1})&=& \sum_{i<j}(-1)^{i+j} c_{0}([\alpha_i,\alpha_j],\dots,\hat{\alpha_i},
\dots,\hat{\alpha_j},\dots,\alpha_{p+1})+\\
&+&\sum_i(-1)^{i+1}L_{\rho(\alpha_i)}(c_{0}(\alpha_1,\dots,\hat{\alpha_i},\dots,\alpha_{p+1})),
\end{eqnarray*}
and then the tail can be computed by applying the principle we have
mentioned above.  The case $k=1$ corresponds to the deformation complex of $A$.

\begin{proposition}\label{proposition k-differentials} The cohomology $H^{\bullet}(A; \Lambda^k\textrm{ad})$ is naturally isomorphic to
the cohomology of the complex $(C^{\bullet}(A;
\Lambda^k\textrm{ad}),d)$. Moreover, the $1$-cocycles of this
complex are precisely the $k$-differentials on $A$.
\end{proposition}

\begin{proof} For the first part, we pick a connection $\nabla$ to realize the adjoint representation and
we claim that the complexes  $(C^{\bullet}(A;
\Lambda^k\textrm{ad}),d)$ and $(\Omega^{\bullet}(A;
\Lambda^k\textrm{ad}),D_{\nabla})$ are isomorphic. For $k=0$ the
statement is trivial while the case $k=1$ follows from Theorem
\ref{The adjoint representation}. The general statement follows from
these two cases if one observes that both $\oplus_kC^{\bullet}(A;
\Lambda^k\textrm{ad})$ and $\oplus_k\Omega^{\bullet}(A;
\Lambda^k\textrm{ad})$ are algebras generated in low degree for
which the corresponding differentials are derivations. For the
second part, remark that an element in $C^1(A;
\Lambda^k\textrm{ad})$ is a pair $(c_0, c_1)$ where $c_0:
\Gamma(A)\to \Gamma(\Lambda^kA)$ and $c_1\in
\Gamma(\Lambda^{k-1}A\otimes TM)$ satisfy the appropriate equation.
Viewing $c_1$ as the map $C^{\infty}(M)\to \Gamma(\Lambda^{k-1}A)$,
$f\mapsto i(df)(c_1)$, we see that the elements of $C^1(A;
\Lambda^k\textrm{ad})$ are precisely the almost $k$-differentials on
$A$. The fact that the cocycle equation is precisely the
$k$-differential equation follows by a simple computation.
\end{proof}
\end{example}

\begin{example}[The coadjoint representation]\rm \  The dual of the adjoint representation of a Lie algebroid $A$
is called the coadjoint representation of $A$, denoted
$\textrm{ad}^{*}$. Using a connection $\nabla$ on $A$, it is given
by the representation up to homotopy
\[ \textrm{ad}^{*}: \underbrace{T^*M}_{\textrm{degree}\ -1}\stackrel{\rho^*}{\to} \underbrace{A^*}_{\textrm{degree}\ 0},\ \ D= \rho^*+ (\nabla^{\textrm{bas}})^*- (R_{\nabla}^{\textrm{bas}})^*.\]
As in the case of the adjoint representation, the resulting
cohomology can be computed by a complex which does
not require the choice of a connection. This complex, denoted
$C^{\bullet}(A; \textrm{ad}^{*})$, is defined as follows. An element
in $C^{p}(A; \textrm{ad}^{*})$ is a pair $c= (c_0, c_1)$ where $c_0$
is a multilinear antisymmetric map
\[ c_0: \underbrace{\Gamma(A)\times \ldots \times \Gamma(A)}_{p\ \textrm{times}}\to \Omega^1(M),\]
and  $c_1\in \Omega^{p-1}(A; A^*)$ is such that
\[ c_{0}(\alpha_1, \ldots , \alpha_{p-1}, f\alpha_p)= f c_{0}(\alpha_1, \ldots , \alpha_{p-1}, \alpha_p)- df\wedge c_{1}(\alpha_1, \ldots , \alpha_{p-1})(\alpha_p),\]
for all $f\in C^{\infty}(M)$, $\alpha_i\in \Gamma(A)$. The
differential of $c$, $d(c)\in C^{p+1}(A; \textrm{ad}^{*})$ is given
by the formulas
\begin{eqnarray*}
(dc)_0(\alpha_1, \dots ,\alpha_{p+1})&=& \sum_{i<j}(-1)^{i+j} c_{0}([\alpha_i,\alpha_j],\dots,
\hat{\alpha_i},\dots,\hat{\alpha_j},\dots,\alpha_{p+1})+\\
&+&\sum_i(-1)^{i+1}L_{\rho(\alpha_i)}(c_{0}(\alpha_1,\dots,\hat{\alpha_i},\dots,\alpha_{p+1})),
\end{eqnarray*}
\begin{eqnarray*}
(dc)_1(\alpha_1, \dots ,\alpha_{p})&=& \sum_{i<j}(-1)^{i+j} c_{0}([\alpha_i,\alpha_j],\dots,
\hat{\alpha_i},\dots,\hat{\alpha_j},\dots,\alpha_{p})+\\
&+&\sum_i(-1)^{i+1}L_{\rho(\alpha_i)}(c_{0}(\alpha_1,\dots,\hat{\alpha_i},\dots,\alpha_{p}))+
(-1)^{p+1}c(\alpha_1, \ldots ,\alpha_p)\circ \rho .
\end{eqnarray*}

\begin{proposition} $(C^{\bullet}(A; \textrm{ad}^{*}), d)$ is a cochain complex whose cohomology is canonically isomorphic to
$H^{\bullet}(A; ad^{*})$.
\end{proposition}
\end{example}

More generally, for any $q$, the representation up to homotopy $S^q(\textrm{ad}^{*})$ and its cohomology
can be treated similarly. This will be made more explicit in our
discussion on the Weil algebra.

\subsection{Cohomology, the derived category, and some more examples}
\label{Cohomology, and some more examples}

As we already mentioned several times, one of the reasons we work with complexes is
that we want to avoid non-smooth vector bundles. The basic idea was
that a complex represents its cohomology bundle (typically a graded
non-smooth vector bundle). To complete this idea, we need to allow
ourselves more freedom when comparing two complexes so that,
morally, if they have the same cohomology bundles, then they become
\textit{equivalent}. This will happen in the derived category. For a
more general discussion on the derived category of a $DG$ algebra we
refer the reader to \cite{Sch}.

\begin{definition} A morphism $\Phi$ between two representations up to homotopy $E$ and $F$
is called a quasi-isomorphism if the first component of $\Phi$, the
map of complexes $\Phi_0: (E, \partial)\to (F, \partial)$, is a
quasi-isomorphism. We denote by $\DDer(A)$ the category obtained
from $\RRep^{\infty}(A)$ by formally inverting the
quasi-isomorphisms, and by $\Der(A)$ the set of isomorphisms classes
of objects of $\DDer(A)$.
\end{definition}
\begin{remark}[Hom in the derived category]
Since we work with vector bundles, there is the following simple
realization of the derived category. Given two representations up to
homotopy $E$ and $F$ of $A$, there is a notion of homotopy between
maps from $E$ to $F$. To describe it, we remark that morphisms in
$\RRep^{\infty}(A)$ from $E$ to $F$ correspond to $0$-cycles in the
complex with coefficients in the induced representation up to
homotopy $\underline{\textrm{Hom}}(E, F)$:
\[ \textrm{Hom}_{\RRep^{\infty}(A)}(E, F)= Z^0(\Omega(A;\underline{\textrm{Hom}}(E,
F))) .\] Two maps $\Phi, \Psi: E\to F$ in $\RRep^{\infty}(A)$ are
called homotopic if there exists a degree $-1$ map $H: \Omega(A;
E)\to \Omega(A; F)$ which is $\Omega(A)$-linear and satisfies $D_EH+
HD_F= \Phi- \Psi$, where $D_E$ and $D_F$ are the structure operators
of $E$ and $F$, respectively. We denote by $[E, F]$ the set of
homotopy classes of such maps. Hence,
\[ [E, F]:= H^0(\Omega(A; \underline{\textrm{Hom}}(E,
F))).\]

As in the case of complexes of vector bundles (see part (2) of Lemma
\ref{lemma-complexes-1}), and by the same type of arguments, we see
that a map $\Phi: E\to F$ is a quasi-isomorphism if and only if it
is a homotopy equivalence. From this, we deduce the following
realization of $\DDer(A)$: its objects are the representations up to
homotopy of $A$, while
\[ \textrm{Hom}_{\DDer(A)}(E, F)= [E, F].\]
Note that, in this language, for any $F\in \RRep^{\infty}(A)$,
\[ H^n(F)= [\mathbb{R}[n], F].\]
Also, the mapping cone construction gives a function
\[ \textrm{Map}: [E, F]\to \Rep^{\infty}(A)\]
which, when applied to $E= \mathbb{R}[n]$, gives the construction
from Example \ref{Differential forms}. Here, the mapping cone $\textrm{Map}(f)$ of a morphism $f: E\rmap F$ 
between representations up to homotopy, is the representation up to 
homotopy whose associated DGA $\Omega(A; \textrm{Map}(f))$ is the 
mapping cone (in the DG sense [19]) of the map $f$, viewed as a DG map 
from $\Omega(A; E)$ to $\Omega(A; F)$.
\end{remark}

\begin{example}\label{derived-ex}\rm \ If $A= \mathcal{F}\subset TM$ is a foliation on $M$, then the
projection from the complex $\mathcal{F}\hookrightarrow TM$
underlying the adjoint representation into $\nu[-1]$ (the normal
bundle $\nu= TM/\mathcal{F}$ concentrated in degree $1$) is clearly
a quasi-isomorphism of complexes. It is easy to see that this
projection is actually a morphism of representations up to homotopy,
when $\nu$ is endowed with the Bott connection (see Example
\ref{zeroth-examples}). Hence, as expected,
\[ \textrm{ad}_{\mathcal{F}}\cong \nu[1] \ \ \ \ (\textrm{in}\ \ \Der(\mathcal{F})).\]
Similarly, for a transitive Lie algebroid $A$, i.e. one for which
the anchor is surjective,
\[ \textrm{ad}_A\cong \mathfrak{g}(A) \ \ \ \ (\textrm{in}\ \ \Der(A)).\]
\end{example}

Next, we will show that the cohomology $H^{\bullet}(A; -)$, viewed
as a functor from the category of representations up to homotopy,
descends to the derived category.

\begin{proposition}\label{qi's} Any quasi-isomorphism $\Phi: E\to F$ between two representations up to homotopy
of $A$ induces an isomorphism in cohomology $\Phi: H(A; E)\to H(A;
F)$.
\end{proposition}

\begin{proof} If $E$ is a representation up to homotopy of  $A$,
one can form a decreasing filtration on $\Omega (A;E)$ induced by
the form-degree
\begin{equation*}
 \cdots \subset F^2(\Omega(A;E))\subset
F^1(\Omega(A,E))\subset F^0(\Omega(A;E))=\Omega(A;E)
\end{equation*}
where
\begin{equation*}
F^p(\Omega(A;E))=\Omega^p(A;E)\oplus \Omega^{p+1}(A;E) \oplus \cdots
\end{equation*}
This filtration induces a spectral sequence $\mathcal{E}$ with
\begin{equation*}
\mathcal{E}^{p,q}_0 = \Omega^p(A; E^q)\Rightarrow H^{p+q}(A;E),
\end{equation*}
where the differential $d^{p,q}_{0}: \mathcal{E}^{pq}_{0}\to
\mathcal{E}^{p, q+1}_{0}$ is induced by the differential $\partial$
of $E$. Given a morphism $\Phi: E\to F$, one has an induced map of
spectral sequences which, at the first level, is induced by $\Phi^0$.
Hence, the assumption that $\Phi$ is a quasi-isomorphism implies
that the map induced at the level of spectral sequences is an
isomorphism at the second level. We deduce that $\Phi$ induces an
isomorphism in cohomology.
\end{proof}

We will now look at the case where $E$ is a regular representation up
to homotopy of $A$ in the sense that the underlying complex $(E,
\partial)$ is regular. In this case, the cohomology
$\mathcal{H}^{\bullet}(E)$ is a graded vector bundle over $M$ (see
Appendix \ref{section-complexes}). The following theorem is an
application of homological perturbation theory (see e.g. \cite{Cra1,GLS}).

\begin{theorem}\label{rep-regular} Let $E$ be a regular representation up to homotopy of $A$.
Then the equation
\begin{equation*}
\overline{\nabla}_{\alpha}([S]):=[\nabla_{\alpha}(S)]
\end{equation*}
for $\alpha\in \Gamma(A)$ and $[S]\in \Gamma( \mathcal{H}^{\bullet}(E))$
makes $\mathcal{H}^{\bullet}(E)$ a representation of $A$. Moreover,
the complex $(\mathcal{H}^{\bullet}(E),0)$ with connection
$\overline{\nabla}$ can be given the structure of a representation
up to homotopy of $A$ which is quasi-isomorphic to $E$.

Also, there is a spectral sequence
\begin{equation*}
\mathcal{E}_2^{pq}=H^p(A;{\mathcal{H}}^q(E))\Rightarrow
H^{p+q}(A,E).
\end{equation*}
\end{theorem}

\begin{proof} The fact that
$\overline{\nabla}$ is flat follows from the fact that the curvature
of $\nabla$ is exact. The spectral sequence is the one
appearing in the previous proof. Next, we have to construct the
structure of representation up to homotopy on
$\mathcal{H}^{\bullet}(E)$. We will use the notations from the proof
of Lemma \ref{lemma-complexes-1} (part 3). The linear Hodge
decomposition $E=ker{\Delta}+im(b)+im(b^*)$ provides
quasi-isomorphisms $p:E\rightarrow {\mathcal{H}}(E)$  and
$i:{\mathcal{H}}(E) \rightarrow E$. The restriction of the Laplacian
to $im(\partial) \oplus im(\partial^*)$, denoted  $\diamondsuit$, is
an isomorphism. Then
\begin{equation*}
h=-\diamondsuit ^{-1}\partial^*
\end{equation*}
satisfies the following equations:
\begin{enumerate}
\item $p\partial = 0,$
\item $\partial i= 0$,
\item $ip=\textrm{Id}+h\partial+\partial h,$
\item $h^2=0,$
\item $ph=0.$
\end{enumerate}
We denote by the same letters the maps induced at the level of
forms, for instance $\partial$ goes from $\Omega(A; E)$ to $\Omega(A; E)$.
We recall that here we use the standard sign conventions. Note that these maps
still satisfy the previous equations. We consider
\[ \delta:= D- \partial: \Omega(A; E)\to \Omega(A; E),\]
where $D$ is the structure operator of $A$. With these,
\[ D_{\mathcal{H}}:=p(1+(\delta h)+(\delta h)^2+(\delta h)^3+\cdots)\delta i : \Omega(A; \mathcal{H})\to\mathrm{} \Omega(A; \mathcal{H}) \]
is a degree one operator which squares to zero, and
\[ \Phi:= p(1+\delta h +(\delta h)^2+ (\delta h)^3 + \cdots): (\Omega(A; E), D)\to (\Omega(A; \mathcal{H}), D_{\mathcal{H}})\]
is a cochain map. The assertions about $D_{\mathcal{H}}$ and $\Phi$
follow by direct computation, or by applying the homological
perturbation lemma (see e.g. \cite{Cra1}), which also explains our
conventions. We now observe that the equations $h^2=0$ and $ph=0$
imply that $\Phi$ is $\Omega(A)$-linear and that $D_{\mathcal{H}}$
is a derivation.
\end{proof}

\begin{example}\rm \ Applying this result to the first quasi-isomorphism discussed
in Example \ref{derived-ex},
we find that the deformation cohomology of a foliation $\mathcal{F}$
is isomorphic to the shifted cohomology of $\mathcal{F}$ with
coefficients in $\nu$- and this is Proposition 4 in \cite{CM}.
\end{example}

\begin{example}[Serre representations]\rm \ Any extension of Lie algebras
\[ \mathfrak{l}\to \tilde{\mathfrak{g}}\to \mathfrak{g}\]
induces a representation up to homotopy of $\mathfrak{g}$ with
underlying complex the Chevalley-Eilenberg complex
$(C^{\bullet}(\mathfrak{l}), d_{\mathfrak{l}})$ of $\mathfrak{l}$.
To describe this representation, we use a splitting $\sigma:
\mathfrak{g}\to \tilde{\mathfrak{g}}$ of the sequence. This induces
\begin{enumerate}
\item For $u\in \mathfrak{g}$, a degree zero operator
\[ \nabla_{u}^{\sigma}:= ad^{*}_{\sigma(u)}: C^{\bullet}(\mathfrak{l})\to C^{\bullet}(\mathfrak{l}),\]
hence a $\mathfrak{g}$-connection $\nabla^{\sigma}$ on
$C^{\bullet}(\mathfrak{l})$.
\item The \textit{curvature} of $\sigma$, $R_{\sigma}\in C^2(\mathfrak{g}; \mathfrak{l})$ given by
\[ R^{\sigma}(u, v)= [\sigma(u), \sigma(v)]- \sigma([u, v]).\]
To produce an
$\underline{\textrm{End}}^{-1}(C^{\bullet}(\mathfrak{l}))$-valued
cochain, we use the contraction operator $i: \mathfrak{l}\to
\underline{\textrm{End}}^{-1}(C^{\bullet}(\mathfrak{l}))$ (but mind
the sign conventions!).
\end{enumerate}
It is now straightforward to check that
\[ D:= d_{\mathfrak{l}}+ \nabla^{\sigma}+ i(R^{\sigma})\]
makes $C^{\bullet}(\mathfrak{l})$ a representation up to homotopy of
$\mathfrak{g}$, with associated cohomology complex isomorphic to
$C^{\bullet}(\tilde{\mathfrak{g}})$.

Note that Theorem \ref{rep-regular} implies that the cohomology of
$\mathfrak{l}$ is naturally a representation of $\mathfrak{g}$, and
there is spectral sequence with
\[ \mathcal{E}_{2}^{p, q}= H^p(\mathfrak{g}; H^q(\mathfrak{l}))\Longrightarrow H^{p+q}(\tilde{\mathfrak{g}}).\]
This is precisely the content of the Serre spectral sequence (see
e.g. \cite{Weibel}).

The same argument applies to any extension of Lie algebroids, giving
us the spectral sequence of Theorem 7.4.6 in \cite{McK}.
\end{example}

In the case of representations of length $1$, we obtain the
following.

\begin{corollary}\label{exact sequence}
If $(E,D)$ is a representation up to homotopy of $A$
 with underlying  regular complex $E=E^0 \oplus E^1$ then there is a long exact
sequence
\begin{equation*}
\cdots \rightarrow H^n(A;{\mathcal{H}}^0(E))\rightarrow
H^n(A;E)\rightarrow H^{n-1}(A;{\mathcal{H}}^1(E)) \rightarrow
H^{n+1}(A;{\mathcal{H}}^0(E))\rightarrow \cdots
\end{equation*}
\end{corollary}
\begin{proof}
The map
\begin{equation*}
H^{n-1}(A;\mathcal{H}^1(E)) \rightarrow H^{n+1}(A;\mathcal{H}^0(E))
\end{equation*}
 is $d_2^{pq}:E_2^{p,1} \rightarrow E_2^{p+2,0}$ in the spectral sequence and is given
by the wedge product with $\omega_2$. The map
\begin{equation*}
H^n(A;\mathcal{H}^0(E))\rightarrow H^n(A;E)
\end{equation*}
is determined by the natural inclusion at the level of cochains.\\
The morphism
\begin{equation*}
H^n(A;E) \rightarrow H^{n-1}(A;\mathcal{H}^1(E))
\end{equation*}
 is given at the level of complexes by:
\begin{equation*}
(\omega_0+ \omega_1) \mapsto \overline{\omega_1}
\end{equation*}
where $\omega_0 \in \Omega^n(A;E^0)$ and $\omega_1 \in
\Omega^{n-1}(A,E^1)$. Since the spectral sequence collapses at the
third stage, we conclude that the sequence is exact.
\end{proof}

\begin{example}\rm \ 
Assume that $A$ is a regular Lie algebroid. We apply Theorem \ref{rep-regular} to the adjoint representation. The
resulting representations in the cohomology are $\mathfrak{g}(A)$ in degree zero and $\nu(A)$ in degree one (see Example \ref{zeroth-examples}). 
Hence, $\textrm{ad}_{\nabla}$ is quasi-isomorphic to a representation up to homotopy on $\mathfrak{g}(A)\oplus \nu(A)[-1]$
where each term is a representation and the differential vanishes, but the $\omega_2$-term (which depends on the connection) does not vanish in general. On the other hand, the previous corollary gives us the 
long exact sequence (\ref{les-def}) (Theorem 3 in \cite{CM}) and we see that the boundary operator $\delta$ is
induced by $\omega_2$.
\end{example}


\section{The Weil algebra  and the BRST model for equivariant cohomology}
\label{The Weil algebra}

In this section we will make use of representations up to homotopy
to introduce the Weil algebra associated to a Lie algebroid $A$,
generalizing the Weil algebra of a Lie algebra and Kalkman's BRST
algebra for equivariant cohomology. Here we give an explicit description
of the Weil algebra which makes use of a connection. An intrinsic description
(without the choice of a connection) is given in \cite{AC2}; 
the price to pay is that the Weil algebra has to be defined as a certain algebra 
of differential operators instead of sections of a vector bundle.

Let $A$ be a Lie algebroid over the manifold $M$ and let 
$\nabla$ be a connection on the vector bundle $A$. Define the algebra
\[ W(A, \nabla)= \bigoplus_{u, v, w} \Gamma(\Lambda^uT^*M\otimes S^{v}(A^*)\otimes \Lambda^{w}(A^*)).\]
It is graded by the total degree $u+ 2v+ w$, and has an underlying
bi-grading $p= v+ w, q= u+v$, so that $W^{p, q}(A, \nabla)$ is the
sum over all $u,v$ and $w$ satisfying these equations. The
connection $\nabla$ will be used to define a differential on $W(A,
\nabla)$. This will be the sum of two differentials
\[ d_{\nabla}= d_{\nabla}^{\textrm{hor}}+ d_{\nabla}^{\textrm{ver}},\]
where, to define $d_{\nabla}^{\textrm{hor}}$ and
$d_{\nabla}^{\textrm{ver}}$ we look at $W(A, \nabla)$ from two
different points of view.

First of all recall that, according to our conventions, the
symmetric powers of the dual of the graded vector bundle
$\mathcal{D}= A\oplus A$ (concentrated in degrees $0$ and $1$)  is
\[ S^{p}\mathcal{D}^*= (\underbrace{\Lambda^pA^*}_{\textrm{degree}\ -p})\oplus (\underbrace{A^*\otimes \Lambda^{p-1}A^*}_{\textrm{degree}\ -p+1})\oplus \ldots \oplus (\underbrace{S^{p-1}A^*\otimes A^*}_{\textrm{degree}\ -1})\oplus (\underbrace{S^{p}A^*}_{\textrm{degree}\ 0}),\]
hence
\[ W^{p, q}(A, \nabla)= \Omega(M; S^{p}\mathcal{D}^*)^{q-p}.\]
We now use the representation up to homotopy structure induced by
$\nabla$ on the double $\mathcal{D}$ of $A$ (see Example \ref{The
double of a vector bundle}), which we dualize and extend to
$S\mathcal{D}^*$. The resulting structure operator will be our
vertical differential
\[ d_{\nabla}^{\textrm{ver}}: W^{p, q}(A, \nabla)\to W^{p, q+1}(A, \nabla).\]

Similarly, for the coadjoint complex $\textrm{ad}^*$ one has
\[ S^{q}\textrm{ad}^{*}= (\underbrace{\Lambda^qT^*M}_{\textrm{degree}\ -q})\oplus (\underbrace{A^*\otimes \Lambda^{q-1}T^*M}_{\textrm{degree}\ -q+1})\oplus \ldots \oplus (\underbrace{T^*M\otimes S^{q-1}A^*}_{\textrm{degree}\ -1})\oplus (\underbrace{S^{q}A^*}_{\textrm{degree}\ 0}),\]
hence
\[ W^{p, q}(A, \nabla)= \Omega(A; S^{q}\textrm{ad}^{*})^{p- q}.\]
We now use the connection $\nabla$ to form the coadjoint
representation $\textrm{ad}_{\nabla}^{*}$. To obtain a horizontal
operator which commutes with the vertical one, we consider the
conjugation of the coadjoint representation, i.e. $\textrm{ad}^{*}$
with the structure operator $-\rho^*+ (\nabla^{\textrm{bas}})^*+
(R_{\nabla}^{\textrm{bas}})^*$. The symmetric powers will inherit a
structure operator, and this will be our horizontal differential
\[ d_{\nabla}^{\textrm{hor}}: W^{p, q}(A, \nabla)\to W^{p+1, q}(A, \nabla) .\]

\begin{proposition} Endowed with $d_{\nabla}^{\textrm{hor}}$ and $d_{\nabla}^{\textrm{ver}}$,
$W(A, \nabla)$ becomes a differential bi-graded algebra whose
cohomology is isomorphic to the cohomology of $M$. Moreover, up to
isomorphisms of differential bi-graded algebras, $W(A, \nabla)$ does
not depend on the choice of the connection $\nabla$.
\end{proposition}

\begin{proof} To prove that the two differentials commute, one first observes that it is enough to check
the commutation relation on functions and on sections of $T^*M$,
$\Lambda^1A^*$ and $S^1A^*$, which generate the entire algebra. This
follows by direct computation (one can also use the local formulas
below, but the computations are much more involved). The
independence of $\nabla$ can be deduced from the fact that, up to isomorphisms, the adjoint 
representation and the double of a vector bundle do not depend on the 
choice of a connection. Alternatively, it follows from the intrinsic 
description of the Weil algebra \cite{AC2} and an argument similar to the one 
in the proof of
Proposition 4.6 (see Proposition 3.9 in {\it loc.cit}). We
need to prove that the cohomology equals that of $M$. Note that for
$p>0$ the column $(W^{p,\bullet},d_{\nabla}^{\textrm{ver}})$ is
acyclic, since it corresponds to an acyclic representation up to
homotopy of $TM$. Next, we note that the first column
$(W^{0,\bullet},d_{\nabla}^{\textrm{ver}})$ is the De-Rham complex
of $M$ and the usual spectral sequence argument provides the desired
result.
\end{proof}

\begin{remark}[Local coordinates] \rm \ Since the operators are local, it is possible to look at their expressions in coordinates.
Let us assume that we are over a chart $(x_a)$ of $M$ on which we
have a trivialization $(e_i)$ of $A$. Over this chart, the Weil
algebra will be the bi-graded commutative algebra over the space of
smooth functions, generated by elements $\partial^a$ of bidegree
$(0, 1)$ ($1$-forms), elements $\theta^i$ of bi-degree $(1, 0)$ (the
dual basis of $(e_i)$, viewed in $\Lambda^1A^*$), and elements
$\mu^i$ of bi-degree $(1, 1)$ (the dual basis of $(e_i)$, viewed in
$S^1A^*$).  A careful but straightforward computation shows that, on
these elements:

\begin{eqnarray*}
d_{\nabla}^{\textrm{ver}}(\partial^a) & =&  0,\\
d_{\nabla}^{\textrm{ver}}(\theta^i) & =& \mu^i-  \Gamma_{aj}^{i}\partial^a\theta^j,\\
d_{\nabla}^{\textrm{ver}}(\mu^i)&=& -  \Gamma_{aj}^{i} \partial^a\mu^j+ \frac{1}{2} r_{abj}^{i}\partial^a\partial^b\theta^j,\\
d_{\nabla}^{\textrm{hor}}(\partial^a)&=& -  \rho^{a}_{i} \mu^i+  (\frac{\partial \rho^{a}_{i}}{\partial x_b}- \Gamma_{bi}^{j}\rho_{j}^{a})\theta^i\partial^b,\\
d_{\nabla}^{\textrm{hor}}(\theta^i)&=& - \frac{1}{2} c^{i}_{jk} \theta^j\theta^k,\\
d_{\nabla}^{\textrm{hor}}(\mu^i)&=& -  (c_{jk}^{i}+ \sum_a
\rho^{a}_{k}\Gamma_{aj}^{i})\theta^j\mu^k+ \frac{1}{2} R_{jka}^{i}
\theta^j\theta^k\partial^a.
\end{eqnarray*}
Here we use Einstein's summation convention. The functions $\rho^{a}_{i}$ are
the coefficients of $\rho$, $c^{i}_{jk}$ are the structure functions
of $A$, $r_{abj}^{i}$ are the coefficients of the curvature of
$\nabla$ and $R_{jka}^{i}$ are the coefficients of the basic
curvature:
\[ \rho(e_i)= \sum \rho^{a}_{i}\partial_a,\ \ [e_j, e_k]= \sum c^{i}_{jk} e_i,\]
\[ R_{\nabla}(\partial_a, \partial_b)e_j= r_{abj}^{i} e_i,\ \  R^{\textrm{bas}}_{\nabla}(e_j, e_k)\partial_a= R_{jka}^{l} e_l.\]
Note that for a smooth function $f$,
\[ d_{\nabla}^{\textrm{ver}}(f)=  \partial_{a}(f) \partial^a, \ \ d_{\nabla}^{\textrm{hor}}(f)=  \partial_{a}(f) \rho^{a}_{i} \omega^i.\]
Also, in case that the connection $\nabla$ is flat, all the $\Gamma$
and $r$-terms above vanish, while the $R$-terms are given by the
partial derivatives of the structure functions $c_{jk}^{i}$.
\end{remark}

\begin{example}[The standard Weil algebra] In the case of Lie algebras $\mathfrak{g}$,
one can immediately see that we recover the Weil algebra
$W(\frak{g})$. In particular, the local coordinates description
becomes
\begin{eqnarray*}
d^{\textrm{ver}}(\theta^i) & = & \mu^i,\\
d^{\textrm{ver}}(\mu^i) & = & 0,\\
d^{\textrm{hor}}(\theta^i) &=&-\frac{1}{2}\sum_{j, k}c^i_{jk}\theta^j\theta^k,\\
d^{\textrm{hor}}(\mu^i)&=&-\sum_{j, k}c^i_{jk}\theta^j\mu^k,
\end{eqnarray*}
which is the usual Weil algebra \cite{Cartan}. In this case,
$d^{\textrm{ver}}$ is usually called the Koszul differential,
denoted $d_{K}$, $d^{\textrm{hor}}$ is called the Cartan
differential, denoted $d_{C}$, and the total differential is denoted
by $d_{W}$.
\end{example}

\begin{example}[The BRST algebra]
Recall Kalkman's BRST algebra \cite{Kalk2} associated to a
$\frak{g}$-manifold $M$. It is $W(\frak{g}, M): = W(\frak{g})\otimes
\Omega(M)$ with differential:
\begin{equation*}
\delta=d_W\otimes 1 + 1\otimes
d_{DR}+\sum_{a=1}^n\theta^a\otimes{\mathcal{L}}_a-\sum_{b=1}^n\mu^b\otimes
\iota_b.
\end{equation*}
\end{example}

\begin{proposition}\label{Weil and BRST} Let $A= \mathfrak{g}\ltimes M$ be the action
algebroid associated to a $\mathfrak{g}$-manifold $M$. Then
\[ W(\mathfrak{g}, M)= W(A; \nabla^{\textrm{flat}}),\]
where $\nabla^{\textrm{flat}}$ is the canonical flat connection on
$A$.
\end{proposition}
\begin{proof}
Follows immediately from the local coordinates description of the
differentials of the Weil algebra.
\end{proof}

\appendix
\section{Appendix}
\subsection{The graded setting}
\label{The graded setting}

Here we collect some general conventions and constructions of graded algebra. As a
general rule, we will be constantly using the \textit{standard sign
convention}: whenever two graded objects $x$ and $y$, say of degrees
$p$ and $q$, are interchanged, one introduces the sign $(-1)^{pq}$.
For instance, the standard commutator $xy- yx$ is replaced by the
graded commutator
\[ [x, y]= xy- (-1)^{pq} yx.\]
Throughout the appendix, $M$ is a fixed manifold and all  vector
bundles are over $M$.

\begin{free}{\bf Graded vector bundles.} \rm \ By a graded vector bundle
over $M$ we mean a vector bundle $E$ together with a direct sum
decomposition indexed by integers:
\[ E= \bigoplus_n E^n.\]
An element $v\in E^n$ is called a homogeneous element of degree $n$
and we write $|v|= n$. Most of the constructions on graded vector
bundles follow by applying point-wise the standard constructions with
graded vector spaces. Here are some of them.
\begin{enumerate}
\item Given two graded vector bundles $E$ and $F$, their direct sum and their tensor product
have natural gradings. On $E\otimes F$ we always use the total
grading:
\[  \textrm{deg}(e\otimes f)= \textrm{deg}(e)+ \textrm{deg}(f).\]
\item Given two graded vector bundles $E$ and $F$ one can also
form the new graded space $\underline{\textrm{Hom}}(E, F)$. Its
degree $k$ part, denoted $\underline{\textrm{Hom}}^k(E, F)$,
consists of vector bundle maps $T: E\to F$ which increase the degree
by $k$.  When $E= F$, we use the notation
$\underline{\textrm{End}}(E)$.
\item For any graded vector bundle, the associated tensor algebra bundle $T(E)$ is graded by the total degree
\[  \textrm{deg}(v_1\otimes \ldots \otimes v_n)=  \textrm{deg}(v_1)+ \ldots + \textrm{deg}(v_n).\]
The associated symmetric algebra bundle $S(E)$ is defined
(fiber-wise) as the quotient of $T(E)$ by forcing $[v, w] = 0$ for
all $v, w\in E$, while for the exterior algebra bundle $\Lambda(E)$
one forces the relations $vw= - (-1)^{pq}wv$ where $p$ and $q$ are
the degrees of $v$ and $w$, respectively.
\item The dual $E^*$ of a graded vector bundle is graded by
\[ (E^*)^n= (E^{-n})^*.\]
\end{enumerate}
\end{free}

\begin{free}{\bf Wedge products.}\rm \ We now discuss wedge products in the graded context. First of all, given
a Lie algebroid $A$ and a graded vector bundle $E$, the space of
$E$-valued $A$-differential forms, $\Omega(A; E)$, is graded by the
total degree:
\begin{equation*}
\Omega(A;E)^p=\bigoplus_{i+j=p}\Omega^i(A;E^j).
\end{equation*}
Wedge products arise in the following general situation. Assume that
$E$, $F$ and $G$ are graded vector bundles and
\[ h: E\otimes F\to G\]
is a degree preserving vector bundle map. Then there is an induced
wedge-product operation
\[ \Omega(A; E)\times \Omega(A; F)\to \Omega(A; G),\ (\omega, \eta)\mapsto \omega\wedge_{h} \eta.\]
Explicitly, for  $\omega\in \Omega^p(A; E^i)$, $\eta\in \Omega^q(A;
F^j)$, $\omega\wedge_h \eta\in \Omega^{p+q}(A; G^{i+j})$ is given by
\[ (\alpha_1,\dots , \alpha_{p+q})\mapsto \sum (-1)^{qi} \textrm{sgn}(\sigma)
h(\omega(\alpha_{\sigma(1)}, \ldots , \alpha_{\sigma(p)}),
\eta(\alpha_{\sigma(p+1)}\ldots , \alpha_{\sigma(p+q)})),\] where
the sum is over all $(p-q)$-shuffles. 
Here is a list of the wedge products that will appear in this paper:
\begin{enumerate}
\item If $h$ is the identity we get:
$$\cdot \wedge\cdot : \Omega(A; E)\otimes \Omega(A; F)\to \Omega(A; E\otimes F).$$
In particular, we get two operations
\begin{equation}
\label{tensor} \Omega(A)\otimes \Omega(A; E)\to \Omega(A; E),\
\Omega(A; E)\otimes \Omega(A)\to \Omega(A; E).
\end{equation}
which make $\Omega(A; E)$ into a (graded) $\Omega(A)$-bimodule. Note
that, while the first one coincides with the wedge product applied
to $E$ viewed as a (ungraded) vector bundle, the second one involves
a sign.
\item If $h$ is the composition of endomorphisms of $E$ we get an operation
\begin{equation}
\label{composition} \cdot \circ \cdot: \Omega(A,
\underline{\textrm{End}}(E))\otimes \Omega(A;
\underline{\textrm{End}}(E))\to \Omega(A,
\underline{\textrm{End}}(E))
\end{equation}
which gives $\Omega(A, \underline{\textrm{End}}(E))$ the structure
of a graded algebra. Of course, this operation makes sense for
general $\underline{\textrm{Hom}}$'s instead of
$\underline{\textrm{End}}$.
\item If $h$ is the evaluation map $\textrm{ev}: \underline{\textrm{End}}(E)\otimes E \to E,\
(T, v)\mapsto T(v)$, we get:
\begin{equation}
\label{left} \cdot \wedge \cdot: \Omega(A;
\underline{\textrm{End}}(E)) \otimes \Omega(A; E)\rightarrow
\Omega(A; E),
\end{equation}
while when $h$ is the twisted evaluation map
$\overline{\textrm{ev}}: E\otimes \underline{\textrm{End}}(E)\to E,
\ (v, T)\mapsto (-1)^{|v||T|} T(v)$, we get:
\begin{equation}
\label{right} \cdot \wedge \cdot: \Omega(A; E)\otimes \Omega(A;
\underline{\textrm{End}}(E))\to \Omega(A; E).
\end{equation}
These operations make $\Omega(A; E)$ a graded $\Omega(A;
\underline{\textrm{End}}(E))$-bimodule.
\item If $h: \Lambda^{\bullet} E\otimes \Lambda^{\bullet}E\to \Lambda^{\bullet} E$ is the multiplication,
we get
$$\cdot \wedge \cdot: \Omega(A; \Lambda^{\bullet} E)\otimes
\Omega(A;\Lambda^{\bullet} E)\to \Omega(A; \Lambda^{\bullet} E)$$
which makes $\Omega(A; \Lambda^{\bullet} E)$ a graded algebra.
\end{enumerate}

Note that the ring $\Omega(A; \underline{\textrm{End}}(E))$ can be
identified with the space of endomorphisms of the left
$\Omega(A)$-module $\Omega(A; E)$ (in the graded sense). More
precisely, we have the following simple lemma.

\begin{lemma}\label{lemma-graded} There is a 1-1 correspondence between degree $n$ elements of $\Omega(A; \underline{\textrm{End}}(E))$ and operators
$F$ on $\Omega(A; E)$ which increase the degree by $n$ and which are
$\Omega(A)$-linear in the graded sense:
\[ F(\omega\wedge \eta)= (-1)^{n |\omega|} \omega\wedge F(\eta)  \ \ \ \forall\ \omega\in \Omega(A), \eta\in \Omega(A; E).\]
Explicitly, $T\in \Omega(A; \underline{\textrm{End}}(E))$ induces
the operator $\hat{T}$ given by:
\[ \hat{T}(\eta)= T\wedge \eta.\]
\end{lemma}

There is one more interesting operation of type
$\wedge_{h}$, namely the one where $h$ is the graded commutator
\[ h: \underline{\textrm{End}}(E)\otimes \underline{\textrm{End}}(E)\to \underline{\textrm{End}}(E),  \ \ h(T, S)= T\circ S- (-1)^{|S||T|} S\circ T.\]
The resulting operation
\[ \Omega(A, \underline{\textrm{End}}(E))\otimes \Omega(A; \underline{\textrm{End}}(E))\to \Omega(A; \underline{\textrm{End}}(E))\]
will be denoted by $[- , -]$. Note that:
\[ [T, S]= T\wedge S- (-1)^{|T||S|}S\wedge T .\]
\end{free}

\subsection{Complexes of vector bundles}
\label{section-complexes}

Here we bring together some
rather standard constructions and facts about complexes of vector bundles.

{\bf Complexes.}\rm \ By a complex over $M$ we mean a cochain
complex of vector bundles over $M$, i.e. a graded vector bundle $E$
endowed with a degree one endomorphism $\partial$
satisfying $\partial^2= 0$:
\[ (E, \partial): \ \ \cdots \stackrel{\partial}{\to}E^{0}\stackrel{\partial}{\to} E^1\stackrel
{\partial}{\to} E^2\stackrel{\partial}{\to} \cdots \] We drop
$\partial$ from the notation whenever there is no danger of
confusion. A morphism between two complexes $E$ and $F$ over $M$ is
a vector bundle map $f: E\to F$ which preserves the degree and is
compatible with the differentials. We denote by $\textrm{Hom}(E, F)$
the space of all such maps. We denote by
$\underline{\textrm{Ch}}(M)$ the resulting category of complexes
over $M$.

\begin{definition} We say that a complex $(E, \partial)$ over $M$ is regular if $\partial$ has constant rank.
In this case one defines the cohomology of $E$ as the graded vector
bundle over $M$:
\[ \mathcal{H}^{\bullet}(E):= \textrm{Ker}(\partial)/\textrm{Im}(\partial).\]
\end{definition}

\begin{remark} Note that $\mathcal{H}^{\bullet}(E)$ only makes sense (as a vector bundle) when $E$ is regular.
On the other hand, one can always take the point-wise cohomology: for
each $x \in M$, there is a cochain complex of vector spaces $(E_x,
\partial_x)$ and one can take its cohomology $H^{\bullet}(E_x,
\partial_x)$. The dimension of these spaces may vary as $x$ varies,
and it is constant if and only if $E$ is regular, in which case they
fit into a graded vector bundle over $M$- and that is $
\mathcal{H}^{\bullet}(E)$.
\end{remark}

As for cochain complexes of vector spaces, we have the following
terminology:
\begin{enumerate}
\item Given two complexes of vector bundles $E$ and $F$ and morphisms $f_1, f_2: E\to F$,
a homotopy between $f$ and $g$ is a degree $-1$ map $h: E \to F$
satisfying
\[ h\partial + \partial h= f_1- f_2.\]
If such an $h$ exists, we say that $f_1$ and $f_2$ are homotopic.
\item A morphisms $f: E\to F$ between two complexes of vector bundles $E$ and $F$ is called
a homotopy equivalence if there exists a morphisms $g: F\to E$ such
that $f\circ g$ and $g\circ f$ are homotopic to the identity maps.
If such an $f$ exists, we say that $E$ and $F$ are homotopy
equivalent. We say that $E$ is contractible if it is homotopy
equivalent to the zero-complex or, equivalently, if there exists a
homotopy between $\textrm{Id}_E$ and the zero map.

\item A morphism $f: E\to F$ between two complexes of vector bundles is called a quasi-isomorphism
if it induces isomorphism in the point-wise cohomologies. We say that
$E$ is acyclic if it is point-wise acyclic.
\end{enumerate}

\begin{lemma}\label{lemma-complexes-1} For complexes of vector bundles over $M$:
\begin{enumerate}
\item[(1)] If $f: E\to F$ is a quasi-isomorphism at $x\in M$, then it is a quasi-isomorphism in
a neighborhood of $x$. In particular, if a complex $E$ is exact at
$x\in M$, then it is exact in a neighborhood of $x$.
\item[(2)] A morphism $f: E\to F$ is a quasi-isomorphism if and only if it is a homotopy equivalence.
In particular, a complex $E$ is acyclic if and only if it is
contractible.
\item[(3)] If a complex $E$ is regular then it is homotopy equivalent to its cohomology
$\mathcal{H}^{\bullet}(E)$ endowed with the zero differential.
\end{enumerate}
\end{lemma}

\begin{proof} For (1) and (2), it suffices to prove the apparently weaker statements in the lemma
coming after ``in particular''. This follows from the standard
mapping complex argument: given a morphism $f$, one builds a double
complex with $E$ as $0$-th row, $F$ as $1$-st row, and $f$ as
vertical differential. The resulting double complex $M(f)$, has the
property that it is acyclic, or contractible, if and only if $f$ is
a quasi-isomorphism, or a homotopy equivalence, respectively \cite{Weibel}.

To prove the weaker statements of (1) and (2), we fix a complex $(E,
\partial)$ and we choose a metric in each vector bundle $E^i$.
Denote $\partial^*$ the adjoint of $\partial$ with respect to the
chosen metric and $$\Delta =
\partial \partial^*+ \partial^* \partial$$ the correspondent ``Laplacian''. It is
not difficult to see that the complex $(E^{\bullet}_x,\partial)$ is
exact if and only if  $\Delta_x$ is an isomorphism. Since the
isomorphisms form an open set in the space of linear
transformations, we get (1). When $(E,\partial)$ is exact, a simple
computation shows that $h:=  \Delta^{-1}\partial^*$ is a contracting
homotopy for $E$, proving (2).

For (3) a linear version of Hodge decomposition gives us
$$E=Ker(\Delta)\oplus Im(\partial)\oplus Im(\partial^*)$$ and an
identification $\mathcal{H}^{\bullet}(E)=Ker(\Delta)$. The resulting
projection $E\rightarrow Ker(\Delta)$ is a quasi-isomorphism.
\end{proof}

{\bf Operations.}\rm \ The operations with graded vector bundles discussed in the
previous section extend to the setting of complexes. In other words,
if $E$ and $F$ are complexes over $M$, then all the associated
graded vector bundles $S(E)$, $\Lambda(E)$, $E^*$,
$\underline{\textrm{Hom}}(E, F)$, $E\otimes F$,  inherit an operator
$\partial$ making them into complexes over $M$. The induced
differentials are defined by requiring that they satisfy the
(graded) derivation rule, written formally as:
\[ \partial (xy)= \partial(x)y+ (-1)^{|x|} x\partial(y).\]
For instance, for $E\otimes F$,
\[ \partial(v\otimes w)= \partial(v)\otimes w+ (-1)^{|v|} v\otimes \partial(w).\]
Also, for $T\in \underline{\textrm{Hom}}(E, F)$,
\[ \partial(T(v))= \partial(T)(v)+ (-1)^{|T|}T(\partial(v)),\]
in terms of graded commutators:
\[ \partial(T)= \partial\circ T- (-1)^{|T|} T\circ \partial= [\partial, T].\]

If $E$ is a complex over $M$, its differential $\partial$ induces a
differential $\partial$ on $\Omega(A; E)$ defined by:
\[ \partial(\eta)= \partial\wedge \eta.\]
Explicitly, for $\eta\in \Omega^p(A; E^k)$, $\partial(\eta)\in
\Omega^p(A; E^{k+1})$ is given by
\[(\alpha_1, \ldots, \alpha_p)\mapsto (-1)^p \partial(\eta(\alpha_1, \ldots, \alpha_p)).\]
The following simple lemma shows that the various differentials
induced on $\Omega(A; \underline{\textrm{End}}(E))$ coincide.

\begin{lemma} For any $T\in \Omega(A; \underline{\textrm{End}}(E))$,
\[ \partial(T)= [\partial, T]= \partial \wedge T- (-1)^{|T|}T\wedge \partial .\]
\end{lemma}

{\bf Connections.}\rm \  Let $A$ be a Lie algebroid. An
$A$-connection on a graded vector bundle $E$ is just an
$A$-connection on the underlying vector bundle $E$ which preserves
the grading. Equivalently, it is a family of $A$-connections, one on
each $E^n$. If $(E, \partial)$ is a complex over $M$, an
$A$-connection on $(E,
\partial)$ is a graded connection $\nabla$ which is compatible with
$\partial$ (i.e. $\nabla_{\alpha}\partial= \partial
\nabla_{\alpha}$). Note that, in terms of the operators $d_{\nabla}$
and $\partial$ induced on $\Omega(A; E)$, the compatibility of
$\nabla$ and $\partial$ is equivalent to $[d_{\nabla}, \partial]=
0$.

Connections on $E$ and $F$ naturally induce connections on the
associated bundles $S(E)$, $\underline{\textrm{End}}(E)$, $E\otimes
F$, etc. The basic principle is, as before, the graded derivation
rule. For instance, one has
\[ d_{\nabla}(\eta_1\wedge \eta_2)= d_{\nabla}(\eta_1)\wedge \eta_2+ (-1)^{|\eta_1|}\eta_1\wedge d_{\nabla}(\eta_2),\]
for all $\eta_1\in \Omega(A; E)$, $\eta_2\in \Omega(A; F)$. Also,
for $T\in \Omega(A; \underline{\textrm{End}}(E))$, $d_{\nabla}(T)$
is uniquely determined by
\[ d_{\nabla}(T\wedge \eta)= d_{\nabla}(T)\wedge \eta+ (-1)^{|T|}T\wedge d_{\nabla}(\eta),\]
for all $\eta\in \Omega(A; E)$. More explicitly,
\[ d_{\nabla}(T)= [d_{\nabla}, T].\]

\begin{lemma}
If a complex $(E,\partial)$ admits an $A$-connection then, for any
leaf $L\subset M$ of $A$, $E|_{L}$ is regular.
\end{lemma}

\begin{proof}
When $A= TM$ there is only one leaf $L= M$, and we have to prove
that $E$ is regular. Since $\nabla$ is compatible with $\partial$,
it follows that the parallel transport with respect to $\nabla$
commutes with $\partial$ and therefore induces isomorphisms between
the point-wise cohomologies. The same argument applied to parallel
transport along $A$-paths, as explained in \cite{CF1}, implies the
general case.
\end{proof}

\end{document}